\documentclass[preprint, 12pt]{elsarticle}
\usepackage[T2A]{fontenc}
\usepackage[cp1251]{inputenc}
\usepackage[russian, english]{babel}
\usepackage{amsmath, amssymb, amsthm}
\usepackage{cmap}

\usepackage[pdftex,left=2cm,right=2cm,top=2cm,bottom=2cm]{geometry}

\usepackage{amsbib}

\usepackage{float}

\usepackage{mathrsfs}

\theoremstyle{plain}
\newtheorem{definition}{Definition}
\newtheorem{theorem}{Theorem}
\newtheorem{corollary}{Corollary}
\newtheorem{lemma}{Lemma}

\theoremstyle{remark}
\newtheorem*{remark*}{Remark}

\DeclareMathOperator{\Span}{span}
\DeclareMathOperator{\Hom}{Hom}
\DeclareMathOperator{\pr}{pr}
\DeclareMathOperator{\id}{id}

\DeclareMathOperator{\grad}{grad}

\newcommand{\ndash}{\nobreakdash-}

\begin{document} 
	
\selectlanguage{english}
\begin{frontmatter}	
	
	
	\title{On holonomy of Weyl connections in Lorentzian signature}
	
	
	\author{Andrei Dikarev}
	
	\address{Department of Mathematics and Statistics, Masaryk University, Faculty of Science, Kotl\'a\v{r}sk\'a 2, 611 37 Brno, Czech Republic\\ Email: xdikareva@math.muni.cz}
	
	\begin{abstract}
		Holonomy algebras of Weyl connections in Lorentzian signature are classified. In particular, examples of Weyl connections with all possible holonomy algebras are constructed.
		\vskip0.5cm	
		\noindent {\bf AMS Mathematics Subject Classification 2020:  53C29; 53C18; 53B30.} 
		
	\end{abstract}
	
	\begin{keyword}
		Weyl connection; Lorentzian signature; holonomy group; holonomy algebra; Berger algebra
	\end{keyword}
	
\end{frontmatter}	

\section{Introduction}
\label{IntroSec}
		
The holonomy group of a connection is an
important invariant. This motivates the classification problem for holonomy
groups. There are classification results for some cases of linear connections. There is a classification of irreducible connected holonomy groups of linear torsion\ndash free connections \cite{Merkulov99}. Important result is the classification of connected holonomy groups of Riemannian manifolds \cite{Berger,Besse87,Bryant87,Joyce07}. Lorentzian holonomy groups are classified  \cite{BerardBergery93,Leistner07,Gal06,Baum12,ESI,GalLorEin,Gal15}. There are partial results for holonomy groups of pseudo-Riemannian manifolds of other signatures \cite{B13,B11,BB-I97,F-K,IRMA,Gal19,Gal18,Gal13,Ik2n,Vol1,Vol2}.

Of certain interests are Weyl manifolds $(M, c, \nabla)$, where 
$c$ is a conformal class of pseudo-Riemannian metrics and $\nabla$ is a torsion-free linear connection preserving $c$. In the Riemannian signature the connected holonomy groups of such connection are classified \cite{Belgun11,Grabbe14}.

The result of this paper is a complete classification of connected holonomy groups (equivalently, of holonomy algebras) of Weyl connections in the Lorentzian signature.

In Sections~\ref{BackgroundSec} we give necessary background on Weyl manifolds and Berger algebras. Berger algebras have the same algebraic properties as the holonomy algebras and they are candidates to the holonomy algebras.
In Section~\ref{MainSec} the main results of the paper are stated.  Theorem~\ref{class2Th} gives the classification of the Berger algebras that preserve proper non-degenerate subspaces of the Minkowski space, these  algebras correspond to conformal products in the sense of \cite{Belgun11}. Then we  assume that the  algebras do not preserve any proper non-degenerate subspace of the Minkowski space, these algebras are called weakly irreducible. If such an algebra is different from $\mathfrak{co}(1,n+1)$, then it preserves an isotropic line. Theorem~\ref{class3Th} gives the classification of such Berger algebras. Theorem \ref{theorem5} states that each obtained Berger algebra is the holonomy algebra of a Weyl connection. The rest of the paper is dedicated to the proofs of the main results.
In Section~\ref{bCasesSec} we describe  weakly irreducible subalgebras of $\mathfrak{co} (1, n + 1)$.
In Section~\ref{AuxSec} we provide the auxiliary results, namely, we found some spaces of algebraic curvature tensors using the first prolongation of Lie algebras representations.  In Section~\ref{CurvTensorSec} we describe the structure of the spaces of curvature tensors for subalgebras of $\mathfrak{co} (1, n + 1)$. The proofs of Theorems~\ref{class2Th} and~\ref{class3Th} are given in Sections~\ref{2ThProofSec} and~\ref{3ThProofSec} correspondingly.
Finally, In Section~\ref{RealizationnSec}, for each Berger algebra $\mathfrak{g} \subset \mathfrak{co}(1, n+1)$,  we construct  a Weyl connection $\nabla$ such that the holonomy algebra of $\nabla$ is  $\mathfrak{g}$, this proves Theorem \ref{theorem5}.

The author would like to express his sincere gratitude to A.~Galaev for useful discussions and suggestions.
The author is grateful to the anonymous referee for careful reading of the paper and valuable comments that greatly affected the final appearance of the paper.
The work was supported by the project MUNI/A/1160/2020.
	
\section{Preliminaries}
\label{BackgroundSec}

Denote by $(M, c)$ a conformal manifold, where $M$ is a smooth manifold, and $c$ is a conformal class of pseudo-Riemannian metrics on $M$. Recall that two metrics $g$ and $h$ are conformally equivalent if and only if $h = e^{2f} g$, for some $f \in C^\infty (M)$.
\begin{definition}
	A Weyl connection $\nabla$ on a conformal manifold $(M, c)$ is a torsion-free linear connection that preserves the conformal class $c$. The triple $(M, c, \nabla)$ is called a Weyl manifold.
\end{definition}
By preserving a conformal class, we understand that if $g \in c$, then there exists a 1-form $\omega_g$ such that
\begin{equation*}
	\nabla g = 2 \omega_g \otimes g.
\end{equation*}	
This formula is conformally invariant in the following sense:
\begin{equation*}
	\text{if}\quad h = e^{2f} g,\enspace f \in C^\infty (M),\quad \text{then}\quad \nabla h = 2 \omega_h \otimes h, \quad \text{where}\quad \omega_h = \omega_g - df.
\end{equation*}	

For the holonomy algebra of a Weyl connection of signature $(r, s)$ we have $\mathfrak{hol}(\nabla) \subset \mathfrak{co} (r, s) = \mathbb{R} \id_{\mathbb{R}^{r,s}} \oplus \mathfrak{so} (r, s)$. If for a metric $g \in c$ it holds ${\nabla g = 0}$, then $\mathfrak{hol}(\nabla) \subset \mathfrak{so} (r, s)$. Then the conformal structure is called closed and we are not interested in this case. Thus we assume that $\mathfrak{hol}(\nabla) \subset \mathfrak{co} (r, s)$ and $\mathfrak{hol}(\nabla) \not\subset \mathfrak{so} (r, s)$. The classification of the holonomy algebras of Weyl connections of positive signature is known~\cite{Belgun11}. 
Namely, let $n = \dim M$, then only the following holonomy algebras of non-closed Weyl structures of Riemannian signature are possible:
\begin{itemize}
	\item $\mathfrak{co}(n)$; 
	\item $\mathbb{R} \id_{\mathbb{R}^{n}} \oplus \mathfrak{so}(k)\oplus \mathfrak{so}(n - k)$, where $1 \leqslant k \leqslant n - 1$;
	\item 
	$\mathbb{R} \id_{\mathbb{R}^{4}} \oplus \left\{ \left. \begin{pmatrix} 
	0 & -a & 0 & 0 \\
	a & 0 & 0 & 0 \\
	0 & 0 & 0 & -a \\
	0 & 0 & a & 0 \end{pmatrix} \right| a \in \mathbb{R} \right\} \quad (n = 4).$
\end{itemize}

\begin{definition}
	Let $\mathfrak{g}$ be a subalgebra of $\mathfrak{gl} (n, \mathbb{R})$. A linear map $R : \mathbb{R}^{n} \wedge \mathbb{R}^{n} \rightarrow \mathfrak{g}$ satisfying the condition
	\begin{itemize}
		\item[] $R(X, Y) Z + R(Y, Z) X + R(Z, X) Y = 0 \quad \forall X, Y, Z \in \mathbb{R}^{n}$ (Bianchi identity)
	\end{itemize}
	is called an algebraic curvature tensor of type $\mathfrak{g}$.
\end{definition}

Let $\mathscr{R} (\mathfrak{g}) $ be the vector space of all algebraic curvature tensors of type $\mathfrak{g}$. Let $$L ( \mathscr{R} (\mathfrak{g} ) ) := \Span \{ R(X, Y) \mid R \in \mathscr{R} (\mathfrak{g}), \quad X, Y \in \mathbb{R}^{n} \} \subset \mathfrak{g}.$$

\begin{definition}
	A subalgebra $\mathfrak{g} \subset \mathfrak{gl} (n, \mathbb{R})$ is called a Berger algebra if $L ( \mathscr{R} (\mathfrak{g} ) ) = \mathfrak{g}$.
\end{definition}

The next theorem follows from the Ambrose-Singer Theorem~\cite{AmbSin53}.

\begin{theorem}
	If $\mathfrak{g} \subset \mathfrak{gl} (n, \mathbb{R})$ is the holonomy algebra of a torsion-free linear connection, then $\mathfrak{g}$ is a Berger algebra.
\end{theorem}

We will denote the Minkowski space by $\mathbb{R}^{1,n+1}$, and by $(\cdot, \cdot)$ the metric on it. A basis $p, e_1, \ldots , e_n, q$ of the space $\mathbb{R}^{1,n+1}$ is called a \textsl{Witt basis} if $p$ and $q$ are isotropic vectors such that $(p, q) = 1$, and $e_1, \ldots , e_n$ is an orthonormal basis of a subspace $\mathbb{R}^{n}$ which is orthogonal to the vectors $p$ and $q$.

Denote by $\mathfrak{co} (1, n + 1)_{\mathbb{R} p}$ the subalgebra of $\mathfrak{co} (1, n + 1)$ preserving the isotropic line $\mathbb{R} p$. It is clear that $$\mathfrak{co} (1, n + 1)_{\mathbb{R} p} = \mathbb{R} \id_{\mathbb{R}^{1,n+1}} \oplus \mathfrak{so} (1, n + 1)_{\mathbb{R} p},$$ where $\mathfrak{so} (1, n + 1)_{\mathbb{R} p}$ is the subalgebra of $\mathfrak{so} (1, n + 1)$ which preserves $\mathbb{R} p$. We can identify the Lie algebra $\mathfrak{so} (1, n + 1)_{\mathbb{R} p}$ with the following matrix Lie algebra:
\begin{equation*}
\mathfrak{so} (1, n + 1)_{\mathbb{R} p} = \left\{
(a, A, X) :=  \left. 
\begin{pmatrix}     
a & X^t & 0 \\
0 & A & -X \\
0 & 0 & -a
\end{pmatrix} \right|
\begin{matrix}     
a \in \mathbb{R} \\
A \in \mathfrak{so} (n) \\
X \in \mathbb{R}^n
\end{matrix} \right\} .
\end{equation*}	
The non-zero brackets in $\mathfrak{so} (1, n + 1)_{\mathbb{R} p}$ are:
\begin{align}
\label{soBrackets}
\begin{split}
&[(a, 0, 0), (0, 0, X)] = (0, 0, aX), \quad [(0, A, 0), (0, 0, X)] =  (0, 0, AX),\\
&[(0, A, 0), (0, B, 0)] = (0, [A, B], 0).
\end{split}
\end{align}
We get the decomposition 
$$\mathfrak{so} (1, n + 1)_{\mathbb{R} p} = (\mathbb{R}\oplus\mathfrak{so}(n))\ltimes\mathbb{R}^n.$$
An element of $\mathfrak{co} (1, n + 1)_{\mathbb{R} p}$ will be denoted by $(b, a, A, X)$, where $b \in \mathbb{R}$ and  $(a, A, X) \in {\mathfrak{so} (1, n + 1)_{\mathbb{R} p}}$.

Recall that each subalgebra $\mathfrak{h}\subset\mathfrak{so}(n)$
is compact and there exists the decomposition
$$
\mathfrak{h}=\mathfrak{h}'\oplus\mathfrak{z}(\mathfrak{h}),
$$
where $\mathfrak{h}'=[\mathfrak{h},\mathfrak{h}]$ is the commutant
of~$\mathfrak{h}$, and $\mathfrak{z}(\mathfrak{h})$ is the center
of~$\mathfrak{h}$~\cite{Onishchik90}.

\begin{definition}
	A Lie subalgebra $\mathfrak{g}\subset \mathfrak{so}(r,s)$ (or $\mathfrak{g}\subset \mathfrak{co}(r,s)$) is called
	weakly irreducible if it does not preserve any proper
	non-degenerate vector subspace of $\mathbb{R}^{r,s}$.
\end{definition}

The following theorem describes weakly irreducible subalgebras of $\mathfrak{so} (1, n + 1)_{\mathbb{R} p}$ and belongs to B\'erard-Bergery and Ikemakhen~\cite{BerardBergery93}.
\begin{theorem}
	\label{soClassTh} 
	A subalgebra $\mathfrak{g}\subset\mathfrak{so}(1,n+1)_{\mathbb{R} p}$ is weakly irreducible if and only if~$\mathfrak{g}$ is a Lie algebra of one of the following
	types.
	
	{\rm\textbf{Type~1}}:
	$$
	\mathfrak{g}^{1,\mathfrak{h}}=(\mathbb{R}\oplus\mathfrak{h})\ltimes
	\mathbb{R}^n=\left\{ \left.\begin{pmatrix}
	a &X^t & 0\\ 0 & A &-X \\ 0 & 0 & -a \\
	\end{pmatrix}\right| a\in \mathbb{R},\ X\in \mathbb{R}^n,\
	A \in \mathfrak{h}\right\},
	$$
	where $\mathfrak{h}\subset\mathfrak{so}(n)$ is a subalgebra.
	
	{\rm\textbf{Type~2}}:
	$$
	\mathfrak{g}^{2,\mathfrak{h}}=\mathfrak{h}\ltimes\mathbb{R}^n=
	\left\{ \left. \begin{pmatrix}
	0 &X^t & 0\\ 0 & A &-X \\ 0 & 0 & 0 \\
	\end{pmatrix}\right| X\in \mathbb{R}^n,\ A \in \mathfrak{h}\right\},
	$$
	where $\mathfrak{h}\subset\mathfrak{so}(n)$ is a subalgebra.
	
	{\rm\textbf{Type~3}}:
	\begin{align*}
	\mathfrak{g}^{3,\mathfrak{h},\varphi}&=\{(\varphi(A),A,0)\mid
	A\in\mathfrak{h}\}\ltimes\mathbb{R}^n
	\\
	&=\left\{ \left. \begin{pmatrix}
	\varphi(A) &X^t & 0\\ 0 & A &-X \\ 0 & 0 & -\varphi(A) \\
	\end{pmatrix}\right| X\in \mathbb{R}^n,\ A \in \mathfrak{h}\right\},
	\end{align*}
	where $\mathfrak{h}\subset\mathfrak{so}(n)$ is a subalgebra
	satisfying the condition $\mathfrak{z}(\mathfrak{h})\ne\{0\}$, and $\varphi\colon\mathfrak{h}\to\mathbb{R}$ is a non-zero linear map
	with the property $\varphi\big|_{[\mathfrak{h}, \mathfrak{h}]}=0$.
	
	{\rm\textbf{Type~4}}:
	\begin{align*}
	\mathfrak{g}^{4,\mathfrak{h},m,\psi}&=\{(0,A,X+\psi(A))\mid
	A\in\mathfrak{h},\ X\in \mathbb{R}^m\}
	\\
	&=\left\{ \left. \begin{pmatrix}
	0 &X^t&\psi(A)^t & 0\\ 0 & A&0 &-X \\ 0 & 0 & 0 &-\psi(A) \\
	0&0&0&0\\ \end{pmatrix} \right| X\in \mathbb{R}^{m},\ A\in
	\mathfrak{h}\right\},
	\end{align*}
	where an orthogonal decomposition
	$\mathbb{R}^n=\mathbb{R}^m\oplus\mathbb{R}^{n-m}$ is fixed,
	$\mathfrak{h}\subset\mathfrak{so}(m)$,
	$\dim\mathfrak{z}(\mathfrak{h})\geqslant n-m$, and
	$\psi\colon\mathfrak{h}\to \mathbb{R}^{n-m}$ is a surjective
	linear map with the property $\psi\big|_{{[\mathfrak{h}, \mathfrak{h}]}}=0$.
\end{theorem}

\section{Main results}
\label{MainSec}

Here we present the classification of the holonomy algebras $\mathfrak{g} \subset \mathfrak{co} (1, n + 1)$ of the Weyl connections such that  $\mathfrak{g} \not\subset \mathfrak{so} (1, n + 1)$, $n \geqslant 0$. 

First suppose that $\mathfrak{g} \subset \mathfrak{co} (1, n + 1)$ is irreducible. Then it is obvious that $\pr_{\mathfrak{so} (1, n + 1)}\mathfrak{g}\subset {\mathfrak{so}(1,n+1)}$ is irreducible as well. Therefore, $\pr_{\mathfrak{so} (1, n + 1)} \mathfrak{g} = \mathfrak{so} (1, n + 1)$, since $\mathfrak{so} (1, n + 1)$ does not have any proper irreducible algebra~\cite{Scala01}. Thus, $\mathfrak{g} = \mathfrak{co} (1, n + 1)$.~\smallskip

Next, let us suppose that $\mathfrak{g}$ preserves a non-degenerate subspace of $\mathbb{R}^{1, n + 1}$.
\begin{theorem}
\label{class2Th}
Let $\mathfrak{g} \subset  \mathfrak{co} (1, n + 1)$, $\mathfrak{g} \not\subset \mathfrak{so} (1, n + 1)$, be a Berger algebra which admits a proper non-degenerate invariant subspace. Then $\mathfrak{g}$ preserves an orthogonal decomposition $$\mathbb{R}^{1, n + 1} = \mathbb{R}^{1, k + 1} \oplus \mathbb{R}^{n - k}, \quad -1 \leqslant k \leqslant n - 1$$ and $\mathfrak{g}$ is conjugated to one of the following subalgebras:
\begin{itemize}
\item $\mathbb{R} \id_{\mathbb{R}^{1,n+1}}\oplus \mathfrak{so} (1, k + 1) \oplus \mathfrak{so} (n - k)  $, $\quad -1 \leqslant k \leqslant n - 1$;
\item $\Big(\mathbb{R}\big(\id_{\mathbb{R}^{1,n+1}}+ (0, -1, 0, 0)\big) \oplus \mathfrak{k} \ltimes \mathbb{R}^{k} \Big) \oplus \mathfrak{so} (n - k) $, where $0 \leqslant k \leqslant n - 1$ and $\mathfrak{k} \subset \mathfrak{so} (k)$ is the holonomy algebra of a Riemannian manifold.
\item $\mathbb{R}\id_{\mathbb{R}^{1,n+1}}\oplus\big(\mathbb{R} (0, -1, 0, 0) \oplus \mathfrak{k} \ltimes \mathbb{R}^{k} \big) \oplus \mathfrak{so} (n - k) $, where $1 \leqslant k \leqslant n - 1$ and $\mathfrak{k} \subset \mathfrak{so} (k)$ is the holonomy algebra of a Riemannian manifold.
\end{itemize}
\end{theorem}

The matrix forms of the last two algebras from Theorem \ref{class2Th} with respect to the basis\\ $p,e_1,\dots,e_{k},q,e_{k+1},\dots ,e_n$ are the following:
$$\left\{ \left. \begin{pmatrix}
0 &X^t&0 & 0\\ 0 & A + b E_k &-X &0 \\ 0 & 0 & 2b &0 \\
0&0&0& B +b E_{n-k}\\ \end{pmatrix} \right| b\in\mathbb{R},\, X\in \mathbb{R}^{k},\ A\in
\mathfrak{k},\, B\in\mathfrak{so}(n-k)\right\},$$
$$\left\{ \left. \begin{pmatrix}
b+a &X^t&0 & 0\\ 0 &  A + b E_k &-X &0 \\ 0 & 0 & b-a &0 \\
0&0&0& B +b E_{n-k}\\ \end{pmatrix} \right| a,b\in\mathbb{R},\, X\in \mathbb{R}^{k},\ A\in
\mathfrak{k},\, B\in\mathfrak{so}(n-k)\right\}.$$

Now, we may assume that $\mathfrak{g}$ does not preserve any non-degenerate subspace of $\mathbb{R}^{1, n+1}$ and it is not irreducible, i.e., it is weakly irreducible and not irreducible. Let $\mathfrak{g}$ preserve a degenerate subspace $W \subset \mathbb{R}^{1, n+1}$. That means that $\mathfrak{g}$ preserves the isotropic line $W \cap W^{\perp}$. We fix a Witt basis $p, e_1, \ldots , e_n, q $ in such a way that $W \cap W^{\perp} = \mathbb{R} p $.

\begin{theorem}
	\label{class3Th} A subalgebra $\mathfrak{g} \subset \mathfrak{co} (1, n + 1)$, $\mathfrak{g} \not\subset \mathfrak{so} (1, n + 1)$, is a weakly irreducible, not irreducible Berger algebra if and only if $\mathfrak{g}$ is conjugated to one of the following subalgebras of the Lie algebra $\mathfrak{co} (1, n + 1)_{\mathbb{R} p}$:
	\begin{itemize}
		\item $\mathbb{R} \id_{\mathbb{R}^{1,n+1}} \oplus \mathfrak{g}^{1,\mathfrak{h}}$, $\mathbb{R} \id_{\mathbb{R}^{1,n+1}} \oplus \mathfrak{g}^{2,\mathfrak{h}}$, $\mathbb{R} \id_{\mathbb{R}^{1,n+1}} \oplus \mathfrak{g}^{3,\mathfrak{h},\varphi}$, where $\mathfrak{g}^{1,\mathfrak{h}}$, $\mathfrak{g}^{2,\mathfrak{h}}$, $\mathfrak{g}^{3,\mathfrak{h},\varphi}$ are from Theorem~\ref{soClassTh}, and $\mathfrak{h} \subset \mathfrak{so} (n)$ is the holonomy algebra of a Riemannian manifold;
		\item $\mathfrak{g}^{\alpha, \theta, 1,\mathfrak{h}} = \{ ( \alpha a + \theta (A), a, A, 0) \mid a \in \mathbb{R}, A \in \mathfrak{h} \} \ltimes \mathbb{R}^{n} = $
		$$\left\{ (\alpha a + \theta (A)) \id_{\mathbb{R}^{1,n+1}} \left.\begin{pmatrix}
		a &X^t & 0\\ 
		0 & A &-X \\ 
		0 & 0 & -a \\
		\end{pmatrix} \right| a\in \mathbb{R},\ X\in \mathbb{R}^n,\
		A \in \mathfrak{h}\right\},
		$$
		where $\alpha\in\mathbb{R}$, $\theta :  \mathfrak{h} \rightarrow \mathbb{R}$ is a linear map such that $\theta \big|_{[\mathfrak{h}, \mathfrak{h}]} = 0$, $\alpha^2 + \theta^2 \neq 0$, and $\mathfrak{h} \subset \mathfrak{so} (n)$ is the holonomy algebra of a Riemannian manifold;
		\item $\mathfrak{g}^{\theta,2,\mathfrak{h}} = \{ (\theta (A), 0, A, 0) \mid A \in \mathfrak{h} \} \ltimes \mathbb{R}^{n} = $
		$$\left\{ \theta (A) \id_{\mathbb{R}^{1,n+1}} + \left.\begin{pmatrix}
		0 &X^t & 0\\ 
		0 & A &-X \\ 
		0 & 0 & 0 \\
		\end{pmatrix} \right| \ X\in \mathbb{R}^n,\
		A \in \mathfrak{h}\right\},
		$$
		where $\theta : \mathfrak{h} \rightarrow \mathbb{R}$ is a non-zero linear map such that $\theta \big|_{[\mathfrak{h}, \mathfrak{h}]} = 0$, and $\mathfrak{h} \subset \mathfrak{so} (n)$ is the holonomy algebra of a Riemannian manifold;
		\item $\mathfrak{g}^{\theta,3,\mathfrak{h},\varphi} = \{ (\theta (A), \varphi (A), A, 0) \mid A \in \mathfrak{h} \} \ltimes \mathbb{R}^{n} = $
		$$\left\{ \theta (A) \id_{\mathbb{R}^{1,n+1}} + \left.\begin{pmatrix}
		\varphi (A) &X^t & 0\\ 
		0 & A &-X \\ 
		0 & 0 & - \varphi (A) \\
		\end{pmatrix} \right| \ X\in \mathbb{R}^n,\
		A \in \mathfrak{h}\right\},
		$$
		where $\theta : \mathfrak{h} \rightarrow \mathbb{R}$ and $\varphi\colon\mathfrak{h}\to\mathbb{R}$ are non-zero linear maps such that $\theta \big|_{[\mathfrak{h}, \mathfrak{h}]} = 0$ and $\varphi\big|_{[\mathfrak{h}, \mathfrak{h}]}=0$, and $\mathfrak{h} \subset \mathfrak{so} (n)$ is the holonomy algebra of a Riemannian manifold.
	\end{itemize}
\end{theorem}


Below, in Section~\ref{RealizationnSec}, for each Berger algebra $\mathfrak{g}$, we construct a Weyl connection with the holonomy algebra isomorphic to $\mathfrak{g}$.

\begin{theorem}\label{theorem5}
	\label{realizationTh}
	Each algebra $\mathfrak{g} \subset \mathfrak{co} (1, n + 1)$ from Theorems~\ref{class2Th} and~\ref{class3Th} is the holonomy algebra of a Weyl connection. 
\end{theorem}

Thus, for non-closed conformal structures of Lorentzian signature we obtained a complete classification of holonomy algebras. These algebras are exhausted by the Lie algebra $\mathfrak{co} (1, n + 1)$ and subalgebras $\mathfrak{g} \subset \mathfrak{co} (1, n + 1)$ from Theorems~\ref{class2Th} and~\ref{class3Th}.

\section{Weakly irreducible subalgebras of $\mathfrak{co} (1, n + 1)_{\mathbb{R} p}$}
\label{bCasesSec}

We will consider the obvious projection
$$\pr_{\mathfrak{so} (1, n + 1)}:\mathfrak{co} (1, n + 1)=\mathbb{R}\oplus\mathfrak{so} (1, n + 1)\to\mathfrak{so} (1, n + 1),$$
which is a homomorphism of the Lie algebras.
For a subalgebra $\mathfrak{g} \subset \mathfrak{co} (1, n + 1)$, denote by $\overline{\mathfrak{g}}$ its projection to $\mathfrak{so} (1, n + 1)$.   It is obvious that $\mathfrak{g}$ is weakly irreducible if and only if $\overline{\mathfrak{g}}$ is weakly irreducible. It is clear that if $\mathfrak{g} $ is weakly irreducible and is not contained in $\mathfrak{so} (1, n + 1)$, then only the following two situations are possible:
\begin{itemize}		
	\item[($a$)]    $\mathfrak{g}$ contains $\mathbb{R} \id_{\mathbb{R}^{1,n+1}}$, in this case    $$\mathfrak{g} = \mathbb{R} \id_{\mathbb{R}^{1,n+1}} \oplus \overline{\mathfrak{g}};$$
	\item[($b$)] $\mathfrak{g}$ does not contain $\mathbb{R}\id_{\mathbb{R}^{1,n+1}}$ and its projection to $\mathbb{R}\id_{\mathbb{R}^{1,n+1}}$ equals $\mathbb{R}\id_{\mathbb{R}^{1,n+1}}$, then
	$$\mathfrak{g} = \{ \theta (B) \id_{\mathbb{R}^{1,n+1}} + B \mid B \in \overline{\mathfrak{g}} \},$$ where $\theta : \overline{\mathfrak{g}} \rightarrow \mathbb{R} $ is a non-zero linear map. 
	It holds $$[\theta (B_1) \id_{\mathbb{R}^{1,n+1}} + B_1,\theta (B_2) \id_{\mathbb{R}^{1,n+1}} + B_2]=[B_1,B_2]\in\mathfrak{g}.$$
	This implies that $\theta$ is a homomorphism of the Lie algebras, and  it holds $\theta \big|_{[\overline{\mathfrak{g}}, \overline{\mathfrak{g}}]} = 0$.
\end{itemize}	

Consider the case ($b$) in more details. Now we describe the map $\theta$ for each Lie algebra from Theorem~\ref{soClassTh}.

\textbf{Case b.1.}
Suppose that $\overline{\mathfrak{g}} = \mathfrak{g}^{1,\mathfrak{h}} = (\mathbb{R} \oplus \mathfrak{h}) \ltimes \mathbb{R}^n$, where $\mathfrak{h} \subset \mathfrak{so} (n)$. 
From~\eqref{soBrackets} it follows that $[\overline{\mathfrak{g}}, \overline{\mathfrak{g}}] = [\mathfrak{h}, \mathfrak{h}] \ltimes \mathbb{R}^n$. Thus we obtain the Lie algebra $$\mathfrak{g} = \{ (\theta (a, A), a, A, 0) \mid a \in \mathbb{R}, A \in \mathfrak{h} \} \ltimes \mathbb{R}^{n},$$ 
where $\theta : \mathbb{R} \oplus \mathfrak{h} \rightarrow \mathbb{R}$ is a non-zero linear map such that $\theta \big|_{[\mathfrak{h}, \mathfrak{h}]} = 0$.

\textbf{Case b.2.}
Let $ \overline{\mathfrak{g}} = \mathfrak{g}^{2,\mathfrak{h}} = \mathfrak{h} \ltimes \mathbb{R}^n$. 
Consider the $\mathfrak{h}$\ndash invariant orthogonal decomposition 
$$\mathbb{R}^{n} = \mathbb{R}^{n_0} \oplus \mathbb{R}^{n_1} \oplus \ldots \oplus \mathbb{R}^{n_r}$$ 
such that $\mathfrak{h}$ acts trivially on $\mathbb{R}^{n_0}$, and $\mathbb{R}^{n_\alpha}$, $\alpha = 1, \ldots , r$, are $\mathfrak{h}$\ndash invariant and the induced representations in them are irreducible. 
Thus, $[\overline{\mathfrak{g}}, \overline{\mathfrak{g}}] = [\mathfrak{h}, \mathfrak{h}] \ltimes (\mathbb{R}^{n_0})^{\perp}$, where $(\mathbb{R}^{n_0})^{\perp} = \mathbb{R}^{n_1} \oplus \ldots \oplus \mathbb{R}^{n_r}$. 
We get
$$\mathfrak{g} = \{ (\theta (X, A), 0, A, X) \mid X \in \mathbb{R}^{n_0}, A \in \mathfrak{h} \} \ltimes (\mathbb{R}^{n_0})^{\perp} ,$$ 
where $\theta : \mathbb{R}^{n_0} \oplus \mathfrak{h} \rightarrow \mathbb{R}$ is a non-zero linear map such that $\theta \big|_{[\mathfrak{h}, \mathfrak{h}]} = 0$.

\textbf{Case b.3.} 
If $\overline{\mathfrak{g}} = \mathfrak{g}^{3,\mathfrak{h},\varphi}$, then as in the case \rm{b.1}, we get
$$\mathfrak{g} = \{ (\theta (A), \varphi (A), A, 0) \mid A \in \mathfrak{h} \} \ltimes \mathbb{R}^{n},$$ 
where $\theta : \mathfrak{h} \rightarrow \mathbb{R}$ is a non-zero linear map such that $\theta \big|_{[\mathfrak{h}, \mathfrak{h}]} = 0$.

\textbf{Case b.4.} 
We will not use the Lie algebra with $\overline{\mathfrak{g}} = \mathfrak{g}^{4,\mathfrak{h}, m, \psi}$, by that reason we will not describe~it.

\section{Auxiliary results}
\label{AuxSec}

Let $\mathbb{R}^{r, s}$ be a pseudo-Euclidean space. We will use the standard isomorphism  $\wedge^2 \mathbb{R}^{r, s}\cong\mathfrak{so} (r, s)$ of the $\mathfrak{so} (r, s)$-modules: $$(X \wedge Y) Z = (X, Z) Y - (Y, Z) X.$$

\begin{definition}
	For a subalgebra $\mathfrak{g} \subset \mathfrak{gl}(n, \mathbb{R})$ its first prolongation is defined as:
	\begin{equation*}
	\mathfrak{g}^{(1)} := \{ \varphi \in \Hom (\mathbb{R}^n, \mathfrak{g}) \mid \varphi (X) Y = \varphi (Y) X \}.
	\end{equation*}
\end{definition}

We are interested in the first prolongations, since below we will see that they are tightly related to algebraic curvature tensors. 
It is well known that $(\mathfrak{so} (r, s))^{(1)} = 0$ and 
\begin{equation}
\label{1prolongEq}
\mathfrak{co} (r, s)^{(1)} = \{ Z \wedge \cdot + (Z, \cdot) \id_{\mathbb{R}^{r,s}} \mid Z \in \mathbb{R}^{r, s} \} \cong \mathbb{R}^{r, s}.
\end{equation}

\begin{lemma}
\label{h1Lemma}
	Let $\mathfrak{h} $ be a subalgebra of $\mathfrak{so} (r, s)$ satisfying one of the following conditions:
	\begin{itemize}
		\item[$\bullet$] $\mathfrak{h} $ is a proper irreducible subalgebra of $\mathfrak{so} (r, s)$;
		\item[$\bullet$] $\mathfrak{h} $ preserves a proper non-degenerate subspace of  $\mathbb{R}^{r, s}$;
		\item[$\bullet$] $\mathfrak{h} $ is an arbitrary proper subalgebra of $\mathfrak{so} (n)$.
	\end{itemize}
	Then  it holds ${ (\mathfrak{h} \oplus \mathbb{R} \id_{\mathbb{R}^{r,s}})^{(1)} = 0 }$.
\end{lemma}

\begin{proof}
	Let $\mathfrak{h} \subset \mathfrak{so} (r, s)$ be an irreducible subalgebra. Suppose that $(\mathfrak{h} \oplus \mathbb{R} \id_{\mathbb{R}^{r,s}})^{(1)} \neq 0$. Since 
	$$(\mathfrak{h} \oplus \mathbb{R} \id_{\mathbb{R}^{r,s}})^{(1)} \subset \mathfrak{co} (r, s)^{(1)} \cong \mathbb{R}^{r, s}$$ 
	is an $\mathfrak{h}$-submodule, then $(\mathfrak{h} \oplus \mathbb{R} \id_{\mathbb{R}^{r,s}})^{(1)} = \mathbb{R}^{r, s}$. From~\eqref{1prolongEq} it follows that $\mathfrak{h} = \mathfrak{so} (r, s)$.
	
	Now suppose that $\mathfrak{h} \subset \mathfrak{so} (r,s)$ preserves a proper non-degenerate subspace $V_1\subset\mathbb{R}^{r,s}$. Then $\mathfrak{h}$ preserves the orthogonal decomposition $$\mathbb{R}^{r,s} = V_1 \oplus V_2,\quad V_2=V_1^\bot.$$ Let $\varphi \in (\mathfrak{h} \oplus \mathbb{R} \id_{\mathbb{R}^{r,s}})^{(1)}$. From~\eqref{1prolongEq} it follows that there exists an element $Z \in \mathbb{R}^{r,s}$ such that
	$$\varphi (X) = Z \wedge X + (Z, X) \id_{\mathbb{R}^{r,s}} $$
	for all $X \in \mathbb{R}^{r,s}$. Let $X \in V_1$ be non-zero. Since $\varphi (X) \in \mathfrak{h} \oplus \mathbb{R} \id_{\mathbb{R}^{r,s}}$, we get $Z\in V_1$. Similarly, taking non-zero $X\in V_2$, we get $Z \in V_2$. This shows that $Z = 0$. 
	
	Finally, any subalgebra $\mathfrak{h}\subset\mathfrak{so}(n)$ is either irreducible, or it preserves a proper  subspace of $\mathbb{R}^n$, which is always non-degenerate.
\end{proof}

\begin{lemma}
\label{coRp1Lemma}
	It holds
	$$(\mathfrak{co} (1, n + 1)_{\mathbb{R} p})^{(1)} = \mathbb{R} (p \wedge \cdot + (p, \cdot) \id_{\mathbb{R}^{1,n+1}}) \cong \mathbb{R} p.$$
\end{lemma}

\begin{proof}
	As we have seen, 
	$$\mathfrak{co} (1, n + 1) ^{(1)} = \{\varphi_Z = Z \wedge \cdot + (Z, \cdot) \id_{\mathbb{R}^{1,n+1}} \mid Z \in \mathbb{R}^{1, n + 1} \}.$$ 
	The Lie algebra $\mathfrak{so} (1, n + 1)_{\mathbb{R} p}$ is spanned by the elements $p \wedge q$, $p \wedge X$ and $X \wedge Y$, where $X, Y \in \mathbb{R}^{n}$. Let $\varphi_Z \in (\mathfrak{so} (1, n + 1) \oplus \mathbb{R} \id_{\mathbb{R}^{1,n+1}} )^{(1)}$ and $X \in \mathbb{R}^{n}$. Then
	\begin{gather*}
	\varphi_Z (X) = Z \wedge X + (Z, X) \id_{\mathbb{R}^{1,n+1}}, \\
	\varphi_Z (q )= Z \wedge q + (Z, q) \id_{\mathbb{R}^{1,n+1}}.
	\end{gather*}
	From the first equation it follows that $Z \in \langle p, e_1, \ldots, e_n\rangle$. From this and the second equation it follows that $Z \in \mathbb{R} p$.
\end{proof}

\begin{corollary}
\label{soRp1Cor}
	Let $\mathfrak{f} \subset \mathfrak{co} (1, n + 1)_{\mathbb{R} p}$ be a subalgebra, then $\mathfrak{f}^{(1)} \neq 0$ if and only if \newline
	$\mathbb{R} (p \wedge q + \id_{\mathbb{R}^{1,n+1}}) \oplus \mathbb{R}^{n} \subset \mathfrak{f}.$
\end{corollary}
\begin{proof}
	Consider a subalgebra $\mathfrak{f} \subset \mathfrak{co} (1, k + 1)_{\mathbb{R} p}$ and assume that $\mathfrak{f}^{(1)} \neq 0$. From the previous lemma it follows that   $\mathfrak{f}^{(1)} = (\mathfrak{co} (1, n + 1)_{\mathbb{R} p})^{(1)}$, and this space is one-dimensional. The image of the basis element of this space coincides with  $\mathbb{R} (p \wedge q + \id_{\mathbb{R}^{1,n+1}}) \oplus \mathbb{R}^{n}$ and it has to be contained in $\mathfrak{f}$. Conversely, if   	$\mathbb{R} (p \wedge q + \id_{\mathbb{R}^{1,n+1}}) \oplus \mathbb{R}^{n} \subset \mathfrak{f},$ then $\mathfrak{f}^{(1)} = (\mathfrak{co} (1, n + 1)_{\mathbb{R} p})^{(1)}$.
\end{proof}

\begin{theorem}
	\label{ssrTh}
	Let $V$ be a pseudo-Euclidean space with a non-trivial orthogonal decomposition $$V = V_1 \oplus V_2.$$ Then the following relation is true:
	\begin{equation*}
	\mathscr{R} (\mathfrak{so} (V_1) \oplus \mathfrak{so} (V_2) \oplus \mathbb{R} \id_V) \cong \mathscr{R} (\mathfrak{so} (V_1)) \oplus \mathscr{R} (\mathfrak{so} (V_2)) \oplus V_1 \otimes V_2.
	\end{equation*}	
\end{theorem}

\begin{proof}
	Take $X_1, Y_1 \in V_1$, $X_2 \in V_2$ and consider the Bianchi identity for $R: \wedge^2 V \rightarrow \mathfrak{so} (V_1) \oplus \mathfrak{so} (V_2) \oplus \mathbb{R} \id_V$:
	\begin{equation}
	\label{BianchiEq}
	R(X_1, Y_1) X_2 +  R(Y_1, X_2) X_1 + R(X_2, X_1) Y_1 = 0.
	\end{equation}
	It follows that $R(X_1, Y_1) X_2 = 0$ and $R(X_1, Y_1) \in \mathfrak{so} (V_1)$. Since the Bianchi identity holds for vectors from $V_1$, then $R \big|_{\wedge^2 V_1} \in \mathscr{R} (\mathfrak{so} (V_1))$. Similarly, $R \big|_{\wedge^2 V_2} \in \mathscr{R} (\mathfrak{so} (V_2))$.
	
	Consider the equality 
	$$R(X_2, Y_1) X_1 = R(X_2, X_1) Y_1.$$
	Fix $X_2$, then it is easy to see, that 
	$$\varphi_{X_2} (\cdot) := R(X_2, \cdot|_{V_1})|_{V_1} \in (\mathfrak{so} (V_1) \oplus \mathbb{R} \id_{V_1} )^{(1)}.$$ Let $R_3 = R \big|_{V_1 \otimes V_2}$, then 
	$${R_3 : V_2 \rightarrow} (\mathfrak{so} (V_1) \oplus \mathbb{R} \id_{V_1})^{(1)} \cong V_1$$
	and $R_3 \in V_1 \otimes V_2$  ($V_1 \otimes V_2$ is irreducible $(\mathfrak{so} (V_1) \oplus \mathfrak{so} (V_2))$-representation).
	
	Thus, $R = R_1 + R_2 + R_3$, where $R_1 = R \big|_{\wedge^2 V_1}$, $R_2 = R \big|_{\wedge^2 V_2}$, and $R_3 = R \big|_{V_1 \otimes V_2 }$. Considering the tensor $R_3 : V_1 \otimes V_2 \rightarrow  \mathfrak{so} (V_1) \oplus  \mathfrak{so} (V_2) \oplus \mathbb{R} \id_V$ and fixing $X_2 \in V_2$, we obtain:
	\begin{equation*}
	R_3 (\cdot \big|_{V_1}, X_2) \big|_{V_1} \in (\mathfrak{so} (V_1) \oplus \mathbb{R} \id_{V_1} )^{(1)}, \quad R_3 (\cdot \big|_{V_1}, X_2) \big|_{V_1} = Z(X_2) \wedge \cdot + (Z(X_2), \cdot) \id_{V_1},
	\end{equation*}
	where $Z : V_2 \rightarrow V_1$ is a linear map. Similarly, by fixing $X_1 \in V_1$, we obtain:
	\begin{equation*}
	R_3 (X_1, \cdot \big|_{V_2}) \big|_{V_2} \in (\mathfrak{so} (V_2) \oplus \mathbb{R} \id_{V_2} )^{(1)}, \quad R_3 (X_1, \cdot \big|_{V_2}) \big|_{V_2} = W(X_1) \wedge \cdot \big|_{V_2} + (W(X_1), \cdot) \id_{V_2},
	\end{equation*}
	where $W : V_1 \rightarrow V_2$ is a linear map.
	
	For arbitrary $X_1, Y_1 \in V_1$, $X_2, Y_2 \in V_2$, we get following:
	\begin{gather*} 
	R_3 (X_1, X_2) Y_1 = (Z(X_2) \wedge X_1) Y_1 + (Z(X_2), X_1) Y_1, \\
	R_3 (X_1, X_2) Y_2 = (W(X_1) \wedge X_2) Y_2 + (W(X_1), X_2) Y_2.
	\end{gather*}
	From the last two equations we conclude that $(Z(X_2), X_1) = (W(X_1), X_2)$, i.e., $W = Z^\ast$. The map $Z : V_2 \rightarrow V_1$ defines $R_3$ by the formula:
	\begin{equation}\label{fR3}
	R_3(X_1, X_2) = Z(X_2) \wedge X_1 + Z^\ast(X_1) \wedge X_2 + (Z(X_2), X_1) \id_V, \quad R_3 \big|_{\wedge^2 V_1} = 0,
	\quad R_3 \big|_{\wedge^2 V_2} = 0.
	\end{equation}
\end{proof}

\begin{corollary}
\label{cor2}
	Let $V$ be a pseudo-Euclidean space with a non-trivial orthogonal decomposition $$V = V_1 \oplus V_2.$$ Let  $\mathfrak{h}_1 \subset\mathfrak{so}(V_1)$ be a subalgebra satisfying one  of the conditions from Lemma \ref{h1Lemma}. Let $\mathfrak{h}_2 \subset \mathfrak{so} (V_2)$ be an arbitrary subalgebra. Then $$\mathscr{R} (\mathfrak{h}_1 \oplus \mathfrak{h}_2 \oplus \mathbb{R} \id_V) = \mathscr{R} (\mathfrak{h}_1 \oplus \mathfrak{h}_2).$$
	In particular, $\mathfrak{h}_1 \oplus \mathfrak{h}_2 \oplus \mathbb{R} \id_V\subset\mathfrak{co}(V)$ is not a Berger algebra.
\end{corollary}
\begin{proof}
	Acting in the same way as in Theorem~\ref{ssrTh}, we get $R = R_1 + R_2 + R_3$, where $R_1 = R \big|_{V_1 \times V_1} \in \mathscr{R} (\mathfrak{h}_1)$, $R_2 = R \big|_{V_2 \times V_2} \in \mathscr{R} (\mathfrak{h}_2)$ and $R_3 = R \big|_{V_1 \otimes V_2 }$. For each $X_2\in V_2$ it holds $$R_3 (\cdot |_{V_1}, X_2)|_{V_1} \in (\mathfrak{h}_1 \oplus \mathbb{R} \id_{V_1} )^{(1)}.$$  According to Lemma~\ref{h1Lemma}, it holds $(\mathfrak{h}_1 \oplus \mathbb{R} \id_{V_1} )^{(1)}=0$, i.e., $R_3=0$.
\end{proof}

\begin{corollary}
\label{sssrCor}
	Let $V$ be a pseudo-Euclidean space with an orthogonal decomposition
	$$V = V_1 \oplus V_2 \oplus V_3,$$ where all subspaces are non-trivial. Then $$\mathscr{R} (\mathfrak{so} (V_1) \oplus \mathfrak{so} (V_2) \oplus \mathfrak{so} (V_3) \oplus \mathbb{R} \id_V) = \mathscr{R} (\mathfrak{so} (V_1)) \oplus \mathscr{R} (\mathfrak{so} (V_2)) \oplus \mathscr{R} (\mathfrak{so} (V_3)).$$
	In particular, $\mathfrak{h}_1 \oplus \mathfrak{h}_2 \oplus \mathfrak{h}_3\oplus \mathbb{R} \id_V\subset\mathfrak{co}(V)$ is not a Berger algebra for arbitrary subalgebras $\mathfrak{h}_i\subset\mathfrak{so}(V_i)$, $i=1,2,3$.
\end{corollary}
\begin{proof}
	The statement follows from Corollary \ref{cor2} applied to the decomposition $V=(V_1 \oplus V_2) \oplus V_3$.
\end{proof}

\section{Algebraic curvature tensors}
\label{CurvTensorSec}

For a subalgebra $\mathfrak{h}\subset\mathfrak{so}(n)$ consider the following space:
\begin{align*}
\nonumber \mathscr{P}(\mathfrak{h})=\{P\in \operatorname{Hom}
(\mathbb{R}^n,\mathfrak{h})\mid{}& g(P(X)Y,Z)+g(P(Y)Z,X)
\\
&+g(P(Z)X,Y)=0,\ X,Y,Z\in \mathbb{R}^n\}
\end{align*}
defined in~\cite{Leistner07}, see also~\cite{Gal10}. The space $\mathscr{P}(\mathfrak{h})$ is called the \textsl{space of weak curvature tensors of type~$\mathfrak{h}$}.

\begin{theorem}
	\label{RcoTh}
	Every algebraic curvature tensor $R \in \mathscr{R} (\mathfrak{co} (1, n + 1)_{\mathbb{R} p})$ is uniquely determined by the elements:
	\begin{gather*} 
	\mu , \lambda \in \mathbb{R}, \quad
	X_0, Z_0 \in \mathbb{R}^n, \quad
	\gamma \in \Hom(\mathbb{R}^n, \mathbb{R}), \quad
	P \in \mathscr{P} (\mathfrak{so} (n)), \quad
	K \in \odot^2 \mathbb{R}^n, \\
	S + \tau \id_{\mathbb{R}^n} \in \mathscr{R} (\mathfrak{co} (n)), \quad \text{where} \quad S \in \Hom(\wedge^2 \mathbb{R}^n,\mathfrak{so} (n)), \quad \tau \in \Hom(\wedge^2 \mathbb{R}^n,\mathbb{R} )
	\end{gather*}
	by the equalities
	\begin{align}
	\label{RcoEq}
	\begin{split}
	R(p, q) &= \big( \mu, \lambda, A_0, X_0 \big) , \\
	R(p, V) &= \big( (Z_0, V), (Z_0, V), Z_0 \wedge V, - (A_0 + \mu E_n) V \big), \\
	R(U, V) &= \big( (A_0 U, V), (A_0 U, V), S(U, V), L(U, V) \big) , \\
	R(U, q) &= \big( \gamma (U), (U, X_0) - \gamma (U), P(U), K(U) \big) , 
	\end{split}
	\end{align}
	where $A_0 \in \mathfrak{so} (n)$ is defined from the condition $\tau (U, V) = (A_0 U, V) $, and 
	$$L(U, V) = P(V) U + \gamma (V) U - P(U) V - \gamma (U) V.$$
\end{theorem}
The idea of the proof is the same as for the proof of a similar theorem from~\cite{Gal05}. The decomposition $\mathbb{R}^{1, n + 1} = \mathbb{R} p \oplus \mathbb{R}^{n} \oplus \mathbb{R} q$ determines the decomposition of the space $\wedge^2 \mathbb{R}^{1, n + 1}$. For $$R :  \wedge^2 \mathbb{R}^{1, n + 1} \rightarrow \mathfrak{co} (1, n + 1)_{\mathbb{R} p} = \mathbb{R} \id_{\mathbb{R}^{1,n+1}} \oplus (\mathbb{R} \oplus \mathfrak{so} (n) \ltimes \mathbb{R}^n)$$ one can consider the restrictions and projections and get various linear operators. It remains to rewrite the Bianchi identity in terms of these operators and solve a problem from linear algebra.

\begin{theorem}
	\label{RhTh}
	 Let $\mathfrak{g} = \mathbb{R} \id_{\mathbb{R}^{1,n+1}} \oplus (\mathbb{R} \oplus \mathfrak{h} ) \ltimes \mathbb{R}^n$, where $\mathfrak{h} \subset \mathfrak{so} (n)$ is a proper subalgebra, and $R \in \mathscr{R} (\mathfrak{g})$, then $R$ satisfies Theorem~\ref{RcoTh} with the following additional constraints:
	\begin{gather*} 
	Z_0 = 0, \quad \tau = 0, \quad A_0 = 0, \quad S \in \mathscr{R} (\mathfrak{h}), \quad P \in \mathscr{P} (\mathfrak{h}).
	\end{gather*}
\end{theorem}

\begin{proof}
Suppose that $\mathfrak{g} = \mathbb{R} \id_{\mathbb{R}^{1,n+1}} \oplus (\mathbb{R} \oplus \mathfrak{h} ) \ltimes \mathbb{R}^n$ and take $R$ as in Theorem~\ref{RcoTh}. The condition $R \in \mathscr{R} (\mathfrak{g})$ is equivalent to the conditions $R \in   \mathscr{R} (\mathfrak{co} (1, n + 1)_{\mathbb{R} p})$ and $R(X, Y) \in \mathfrak{g}$ for all $X, Y \in \mathbb{R}^{1, n + 1}$. It follows that
\begin{equation*}
A_0 \in \mathfrak{h}, \quad
Z_0 \wedge V \in \mathfrak{h}, \quad
S(U, V) \in \mathfrak{h}, \quad
P \in \mathscr{P} (\mathfrak{h}) \quad 
\text{for all} \quad
U, V \in \mathbb{R}^{n}.
\end{equation*}

\textbf{Case 1.} Suppose that $\mathfrak{h} \subset \mathfrak{so} (n)$ is a proper irreducible subalgebra. Note that 
$$Z_0 \wedge V + (Z_0, V) \id_{\mathbb{R}^{1,n+1}} \in (\mathfrak{h} \oplus \mathbb{R} \id_{\mathbb{R}^{1,n+1}})^{(1)}.$$
According to Lemma~\ref{h1Lemma}, $(\mathfrak{h} \oplus \mathbb{R} \id_{\mathbb{R}^{1,n+1}})^{(1)} = 0$, therefore $Z_0 = 0$. Since $S(U, V) \in \mathfrak{h}$ we get, that $S + \tau \id_{\mathbb{R}^n} \in \mathscr{R} (\mathfrak{h} \oplus \mathbb{R} \id_{\mathbb{R}^n})$. From the classification from~\cite{Merkulov99} it follows that $\mathscr{R} (\mathfrak{h} \oplus \mathbb{R} \id_{\mathbb{R}^n}) = \mathscr{R} (\mathfrak{h})$. Hence $\tau = 0$ and $A_0 = 0$.

\textbf{Case 2.} Suppose that $\mathfrak{h} $ is not irreducible. Then there exists an  $\mathfrak{h}$-invariant decomposition $\mathbb{R}^{n} = \mathbb{R}^{k} \oplus \mathbb{R}^{n - k}$. Let $\mathfrak{h} = \mathfrak{so} (k) \oplus \mathfrak{so} (n - k)$. Then, according to Theorem~\ref{ssrTh}:
$$(S + \tau \id_{\mathbb{R}^{n}})(X_1, X_2) = Z(X_2) \wedge X_1 + Z^{\ast}(X_1) \wedge X_2 + (Z(X_2), X_1) \id_{\mathbb{R}^{n}}.$$ 
Now, 
$$\tau (X_1, X_2) = (Z(X_2), X_1) = (A_0 X_1, X_2)$$
and, hence, $A_0 X_1 = Z^{\ast}(X_1)$. Also, from the equality $$(A_0 X_2, X_1) = - (A_0 X_1, X_2) = - (Z(X_2), X_1),$$ we obtain $A_0 X_2 = - Z(X_2)$. But $A_0 \in \mathfrak{h}$, consequently $Z = 0$ and so $A_0 = 0$ and $\tau = 0$. The same result is true for an arbitrary subalgebra of $\mathfrak{so} (k) \oplus \mathfrak{so} (n - k)$.
\end{proof}

\begin{theorem}
	\label{prsoTh}
	Let $\mathfrak{g} \subset \mathfrak{co} (1, n + 1)_{\mathbb{R} p}$ be a subalgebra and let $\mathfrak{h} = \pr_{\mathfrak{so} (n)} \mathfrak{g} $. If $\mathfrak{g} $ is a Berger subalgebra, then $\mathfrak{h} \subset \mathfrak{so} (n) $ is the holonomy algebra of a Riemannian manifold.
\end{theorem}

\begin{proof}
If $\mathfrak{h} = \mathfrak{so} (n) $, then $\mathfrak{h}$ is the holonomy algebra of a Riemannian manifold. Now suppose that $\mathfrak{h} \not\subset \mathfrak{so} (n) $, in this case, from Theorem~\ref{RhTh} it follows that $A_0 = 0$ and $Z_0 = 0$. Thus $\mathfrak{h}$ is generated by the images of the elements $S \in \mathscr{R} (\mathfrak{h})$, $P \in \mathscr{P} (\mathfrak{h})$. Therefore $\mathfrak{h}$ is the holonomy algebra of a Riemannian manifold according to~\cite{Leistner07}.
\end{proof}

\begin{remark*}
In what follows we will use the following fact.
Suppose that $\mathfrak{g} \subset \mathfrak{co}(1,n+1)_{\mathbb{R}p}$ is a subalgebra, then
$R\in\mathscr{R}(\mathfrak{g})$ if and only if  $R \in   \mathscr{R} (\mathfrak{co} (1, n + 1)_{\mathbb{R} p})$ and $R(X, Y) \in \mathfrak{g}$ for all $X, Y \in \mathbb{R}^{1, n + 1}$. This means that $R$ may be described as in Theorem \ref{RcoTh} with the additional condition that $R$ takes values in $\mathfrak{g}$.
If the projection $\mathfrak{h}$ of $\mathfrak{g}$ to $\mathfrak{so}(n)$ is a proper subalgebra of $\mathfrak{so}(n)$, then Theorem~\ref{RhTh} may be used.
\end{remark*}

\section{Proof of Theorem~\ref{class2Th}}
\label{2ThProofSec}

By the assumption of the theorem we have a $\mathfrak{g}$-invariant decomposition 
$$\mathbb{R}^{1, n + 1} = \mathbb{R}^{1, k + 1} \oplus \mathbb{R}^{n - k},$$ 
i.e., 
$$\mathfrak{g} \subset \mathfrak{so} (1, k + 1) \oplus \mathfrak{so} (n - k) \oplus \mathbb{R} \id_{\mathbb{R}^{1,n+1}}.$$ 
We will consider two cases.

{\bf Case 1.}
First suppose, that $\pr_{\mathfrak{so} (1, k + 1)} \mathfrak{g}$ is irreducible. 
Then, $\pr_{\mathfrak{so} (1, k + 1)} \mathfrak{g} = \mathfrak{so} (1, k + 1)$ and $k \geqslant 1$. 
Consider the ideal $\mathfrak{a} = \mathfrak{g} \cap \mathfrak{so} (n - k) \subset \pr_{\mathfrak{so} (n - k)} \mathfrak{g}$. 
Since each subalgebra of $\mathfrak{so} (n - k)$ is  reductive, there exists an ideal $\mathfrak{b} \subset \pr_{\mathfrak{so} (n - k)} \mathfrak{g}$
such that $$\pr_{\mathfrak{so} (n - k)} \mathfrak{g}=\mathfrak{a}\oplus \mathfrak{b}.$$ By the definition of the ideal $\mathfrak{a}$,  the projection $\pr_{\mathfrak{so} (1, n + 1)} \mathfrak{g}$ may be described in the following way.  There exists a surjective linear map $$\varphi : \mathfrak{so} (1, k + 1) \rightarrow \mathfrak{b}$$ such that 
$$\pr_{\mathfrak{so} (1, n + 1)}\mathfrak{g} = \{A + \varphi(A) \mid A \in \mathfrak{so} (1, k + 1) \} \oplus \mathfrak{a}.$$ 
It holds $$[A_1 + \varphi(A_1),A_2 + \varphi(A_2)]=[A_1,A_2] + [\varphi(A_1),\varphi(A_2)]\in\mathfrak{so} (1, k + 1) \oplus\mathfrak{b}.$$
This implies that $\varphi$ is a homomorphism of Lie algebras. 
The Lie algebra $\mathfrak{so} (1, k + 1)$ is simple for $k \geqslant 1$, so either $\ker \varphi = 0$ or $\varphi = 0$. 
If $\ker \varphi = 0$, then $\varphi$ is an isomorphism. This is impossible, since $\mathfrak{b} \subset \mathfrak{so} (n - k)$ is a compact Lie algebra, while $\mathfrak{so} (1, k + 1)$ is not compact. 
Thus, $\varphi = 0$ and, hence, $$\pr_{\mathfrak{so} (1, n + 1)}\mathfrak{g} = \mathfrak{so} (1, k + 1) \oplus \mathfrak{a}.$$ If $\mathfrak{a} \subset \mathfrak{so} (n - k)$ is a proper subalgebra, then by Corollary~\ref{cor2}, $\mathfrak{g}$ is not a Berger algebra. 
Therefore, $\mathfrak{a}=\mathfrak{so} (n - k)$. If $\mathfrak{g}$ contains $\mathbb{R} \id_{\mathbb{R}^{1,n+1}}$, then
$$\mathfrak{g} = \mathfrak{so} (1, k + 1) \oplus \mathfrak{so} (n - k) \oplus \mathbb{R} \id_{\mathbb{R}^{1,n+1}}.$$
Suppose that $\mathfrak{g}$ does not contain $\mathbb{R} \id_{\mathbb{R}^{1,n+1}}$. Using arguments as in Section \ref{bCasesSec}, we see that this is possible only for $k=n-2$ (i.e., when $\pr_{\mathfrak{so} (1, n + 1)}\mathfrak{g}$ is not semisimple), and then
$$\mathfrak{g} = \mathfrak{so} (1, k + 1) \oplus \mathbb{R}\big(e_{n-1}\wedge e_n+ a\id_{\mathbb{R}^{1,n+1}}\big)$$
for some non-zero $a\in\mathbb{R}$. Let $R\in\mathscr{R}(\mathfrak{g})$. Then it may be described as in the proof of Theorem~\ref{ssrTh}. Let $X_1\in V_1$.
Using~\eqref{fR3}, we obtain
$$ R_3(X_1, e_{n-1}) = Z(X_2) \wedge X_1 + Z^\ast(X_1) \wedge e_{n-1} + (Z(e_{n-1}), X_1) \id_{\mathbb{R}^{1,n+1}},$$
$$ R_3(X_1, e_{n}) = Z(X_2) \wedge X_1 + Z^\ast(X_1) \wedge e_{n} + (Z(e_{n}), X_1) \id_{\mathbb{R}^{1,n+1}}.$$
Note that $Z^\ast(X_1) \wedge e_{n-1}=(Z^\ast(X_1),e_n) e_n\wedge e_{n-1}$, and the first equality implies 
$$-a(Z^\ast(X_1),e_n)=(Z(e_{n-1}), X_1).$$
Consequently,
$$-a(Z(e_n),X_1)=(Z(e_{n-1}), X_1).$$
Likewise, the second equation implies 
$$a(Z(e_{n-1}),X_1)=a(Z(e_{n}), X_1).$$
We conclude that
$$(Z(e_n),X_1)=-a^2(Z(e_n),X_1),$$
i.e., $Z=0$. Thus, $R=R_1$ takes values in $\mathfrak{so}(1,k+1)$, and $\mathfrak{g}$ is not a Berger algebra.
 
{\bf Case 2.}
Now suppose that $\pr_{\mathfrak{so} (1, k + 1)} \mathfrak{g}$ is not irreducible. If $\pr_{\mathfrak{so} (1, k + 1)} \mathfrak{g}$ preserves a proper non-degenerate subspace of $\mathbb{R}^{1,k+1}$, then $\mathfrak{g}$ preserves an orthogonal direct sum of three non-trivial subspaces of $\mathbb{R}^{1,n+1}$, and
from  Corollary~\ref{sssrCor} it follows that $\mathfrak{g}$ is not a Berger algebra.  
Thus, $\pr_{\mathfrak{so} (1, k + 1)} \mathfrak{g}\subset\mathfrak{so} (1, k + 1)$ is weakly irreducible and it is not irreducible. Hence $\mathfrak{g}$  preserves an isotropic line in $\mathbb{R}^{1, k + 1}$. 
Choose a Witt basis $p, e_1, \ldots , e_n, q $ in such a way, that $\pr_{\mathfrak{so} (1, k + 1)} \mathfrak{g} \subset \mathfrak{so} (1, n + 1)_{\mathbb{R} p}$, then $$\mathfrak{g} \subset \mathfrak{so} (1, k + 1)_{\mathbb{R} p} \oplus \mathfrak{so} (n - k) \oplus \mathbb{R} \id_{\mathbb{R}^{1,n+1}}.$$
We assume that $\mathbb{R}^{1, k + 1}$ and $\mathbb{R}^{n - k}$ have bases $p, e_1, \ldots , e_k, q $ and $e_{k+1}, \ldots , e_n$, correspondingly. 
Consider $R \in \mathscr{R} (\mathfrak{g})$ as in the proof of Theorem~\ref{ssrTh}. 
It holds $R = R_1 + R_2 + R_3$. 
Note that $R_1$ takes values in $\mathfrak{g} \cap \mathfrak{so} (1, k + 1)$, and $R_2$ takes values in $\mathfrak{g} \cap \mathfrak{so} (n - k)$. 
By Lemma~\ref{coRp1Lemma} and the proof of Theorem~\ref{ssrTh}, the image of the map $Z : \mathbb{R}^{n - k} \rightarrow \mathbb{R}^{1, k + 1}$ in the formula for $R_3$ is contained in $\mathbb{R} p$, i.e., there exists a linear map $\alpha : \mathbb{R}^{n - k} \rightarrow \mathbb{R}$ such that $Z(X_2) = \alpha (X_2) p$, where $X_2 \in \mathbb{R}^{n - k}$. 
From the relation
\begin{equation*}
(Z^{\ast}(X_1), X_2) = (X_1, Z(X_2)) = (X_1, \alpha(X_2) p) = \alpha(X_2)(X_1,  p)
\end{equation*}	
it follows that $Z^{\ast}(p) = 0$, $Z^{\ast}(e_i) = 0$, $i = 1, \ldots, k$. 
And from the formula
\begin{equation*}
R_3 (X_1, X_2) = \alpha(X_2) p \wedge X_1 + Z^{\ast}(X_1) \wedge X_2 + \alpha(X_2)(X_1,  p) \id_{\mathbb{R}^{1,n+1}}
\end{equation*}	
we get
\begin{align*}
\begin{split}
&R_3 (p, X_2) = 0, \\
&R_3 (e_i, X_2) = \alpha(X_2) p \wedge e_i, \quad 1 \leqslant k \leqslant n,\\
&R_3 (q, X_2) = \alpha(X_2) p \wedge q + Z^{\ast}(q) \wedge X_2 + \alpha(X_2) \id_{\mathbb{R}^{1,n+1}}.
\end{split}
\end{align*}
Since we assume that $\mathfrak{g} \not\subset \mathfrak{so} (1, n + 1)$ and that $\mathfrak{g}$ is a Berger algebra, then there exists $R_3 \in \mathscr{R} (\mathfrak{g})$ such that $\alpha \neq 0$, and, hence, $p \wedge \mathbb{R}^{k} \subset \mathfrak{g}$.

For convenience, suppose that $\alpha (e_{k+1}) = a \neq 0$, ${\alpha (e_{k+2}) = \ldots = \alpha (e_{n}) = 0}$. In this case, $Z^{\ast}(q) = a e_{k+1}$ and we obtain the following relations:
\begin{align*}
\begin{split}
&R_3 (q, e_{k+1}) = a(p \wedge q + \id_{\mathbb{R}^{1,n+1}}) \in \mathfrak{g}, \\
&R_3 (q, e_j) =a e_{k+1} \wedge e_j \in \mathfrak{g}, \quad k + 2 \leqslant j \leqslant n.
\end{split}
\end{align*}
We conclude that $p \wedge q + \id_{\mathbb{R}^{1,n+1}}\in\mathfrak{g}$. We claim also that $\mathfrak{g}$ contains $\mathfrak{so} (n - k)$. If $n-k=1$, there is nothing to prove. Suppose that $n-k\geq 2$.
If $j, i \neq k + 1$, then $$[e_{k+1} \wedge e_j, e_{k+1} \wedge e_i] = e_j \wedge e_i \in \mathfrak{g}.$$ This proves the claim.

Thus, $$\mathfrak{g} \subset \mathbb{R} \id_{\mathbb{R}^{1,n+1}} \oplus \mathbb{R}p\wedge q \oplus \mathfrak{so} (k) \oplus \mathfrak{so} (n - k) \ltimes p\wedge\mathbb{R}^{k},$$ and  $\mathfrak{g}$ contains $p \wedge \mathbb{R}^{k}$, $\mathfrak{so} (n - k)$ and $\mathbb{R} ( p \wedge q + \id_{\mathbb{R}^{1,n+1}} )$. 
Consider the projection $\pr_{\mathfrak{so} (1, k + 1)} \mathfrak{g} \subset \mathfrak{so} (1, k + 1)_{\mathbb{R} p}$, it is weakly irreducible, contains $p \wedge \mathbb{R}^{k}$ and  $\mathbb{R} p \wedge q$. 
Therefore, $\pr_{\mathfrak{so} (1, k + 1)} \mathfrak{g}$ is of type~1, i.e., $$\pr_{\mathfrak{so} (1, k + 1)} \mathfrak{g} = \mathbb{R}p\wedge q \oplus \mathfrak{k} \ltimes p\wedge \mathbb{R}^{k},$$ where $\mathfrak{k}\subset\mathfrak{so}(k)$ is a subalgebra. According to Theorem~\ref{prsoTh}, $\mathfrak{k} \subset \mathfrak{so} (k)$ is the holonomy algebra of a Riemannian manifold.

Thus, $$\pr_{\mathfrak{so} (1, n + 1)} \mathfrak{g}=
\mathbb{R}p\wedge q \oplus \mathfrak{k}\oplus\mathfrak{so}(n-k) \ltimes p\wedge \mathbb{R}^{k},$$
and $\mathfrak{g}$ contains  $\mathbb{R} ( p \wedge q + \id_{\mathbb{R}^{1,n+1}} )$.
Using arguments of Section \ref{bCasesSec}, it is easy to see that this is the case only for the following two algebras:

$$\mathfrak{g}=
\mathbb{R}\id_{\mathbb{R}^{1,n+1}}\oplus\mathbb{R}p\wedge q \oplus \mathfrak{k}\oplus\mathfrak{so}(n-k) \ltimes p\wedge \mathbb{R}^{k},$$
$$\mathfrak{g}=
\mathbb{R}(\id_{\mathbb{R}^{1,n+1}}+p\wedge q) \oplus\{\theta(A)\id_{\mathbb{R}^{1,n+1}}+A|A\in  \mathfrak{k}\oplus\mathfrak{so}(n-k)\} \ltimes p\wedge \mathbb{R}^{k},$$
where $\theta:\mathfrak{k}\oplus\mathfrak{so}(n-k)\to\mathbb{R}$ is a linear map vanishing on the commutant of $\mathfrak{k}\oplus\mathfrak{so}(n-k)$.

If $k\geq 1$, then the first algebra is the third algebra from the statement of the theorem. If $k=0$, 
then the first algebra is the first algebra from the statement of the theorem.

We claim that the second algebra is a Berger algebra only if $\theta=0$. Indeed,
again we may decompose each $R\in\mathscr{R}(\mathfrak{g})$ as the sum $R=R_1+R_2+R_3$. The tensor $R_3$ may be non-zero only  on bivectors from $\mathbb{R}^{1,k+1}\wedge \mathbb{R}^{k}$ and it takes values in 
$\mathbb{R}(\id_{\mathbb{R}^{1,n+1}}+p\wedge q) \oplus\mathfrak{so}(n-k)\ltimes p\wedge \mathbb{R}^{k}$;
the tensor $R_1$ may be non-zero only  on bivectors from $\Lambda^2\mathbb{R}^{1,k+1}$ and it takes values in $\mathfrak{g}\cap \mathfrak{so}(1,k+1)$; $R_1$ may be non-zero only  on bivectors from $\Lambda^2\mathbb{R}^{n-k}$ and it takes values in $\mathfrak{g}\cap \mathfrak{so}(n-k)$. This implies that if $\theta(A)\neq 0$ for some $A$, then $\theta(A)+A$ does not belong to $L(\mathscr{R}(\mathfrak{g}))$. Thus, $\theta=0$.
The obtained  algebra is the second algebra from the statement of the theorem.

The theorem is proved.
\hfill\qedsymbol
	
\section{Proof of Theorem~\ref{class3Th}}
\label{3ThProofSec}

For convenience, we will use the notation of Theorem~\ref{RcoTh} for the components of an algebraic curvature tensor~$R$. 
Consider the cases as in Section~\ref{bCasesSec}.

{\bf Case a.}
Let $\mathfrak{g} = \mathbb{R} \id_{\mathbb{R}^{1,n+1}} \oplus \overline{\mathfrak{g}}$, where $\overline{\mathfrak{g}}$ of the type 1, 2 or 3 from Theorem~\ref{soClassTh}. 
If $\mathfrak{g}$ is a Berger algebra, then according to Theorem~\ref{prsoTh}, $\mathfrak{h}$ is the holonomy algebra of a Riemannian manifold. 
Consider the algebraic curvature tensor $R$, which is defined by the condition $\mu = 1$ and all other elements are zeros. 
Since ${R \in \mathscr{R} (\mathfrak{g})}$, it holds $\mathbb{R} \id_{\mathbb{R}^{1,n+1}} \subset L(\mathscr{R} (\mathfrak{g}))$. 
Since $\overline{\mathfrak{g}}$ is a Berger algebra and $\mathfrak{g} = \mathbb{R} \id_{\mathbb{R}^{1,n+1}} \oplus \overline{\mathfrak{g}}$, $\mathfrak{g}$ is a Berger algebra.

Now assume that $\mathfrak{g} = \mathbb{R} \id_{\mathbb{R}^{1,n+1}} \oplus \overline{\mathfrak{g}}$, where $\overline{\mathfrak{g}}$ is of type 4. Let $R \in \mathscr{R} (\mathfrak{g})$. Since $\mathfrak{h} \subset \mathfrak{so} (m)$, it is a proper subalgebra of $\mathfrak{so}(n)$. As we remarked in Section \ref{CurvTensorSec}, $R$ may be described in terms of Theorem~\ref{RhTh}.
In particular, $A_0 = 0$ and $Z_0 = 0$. 
Choose an arbitrary non-zero element $V \in \mathbb{R}^{n - m}$. It holds $$R(p, V) = (0, 0, 0, - \mu V).$$ Since $-\mu V = \psi (0) = 0$, it holds $\mu = 0$.

We are going to show that $\gamma = 0$. 
Since $X_0 \in \mathbb{R}^{m}$, then $\gamma (U) = (U, X_0) = 0$ for every $U \in \mathbb{R}^{n - m}$, i.e., $\gamma \big|_{\mathbb{R}^{n - m}} = 0$. 
Now we take $U \in \mathbb{R}^{m}$, $V \in \mathbb{R}^{n - m}$ and arbitrary $X, Y \in \mathbb{R}^{n}$, then for $S \in \mathscr{R} (\mathfrak{h}) $, the formula 
$$(S(U, V) X, Y) = (S(X, Y) U, V)$$ 
holds. 
Since $(S(X, Y) U, V) = 0$, it holds $S(U, V) = 0$. 
It follows that 
$$\pr_{\mathbb{R}^{n - m}} (L(U, V)) = \psi (S(U, V)) = 0,$$ but 
$$L(U, V) = P(V) U + \gamma (V) U - P(U) V - \gamma (U) V,$$ therefore $\pr_{\mathbb{R}^{n - m}} (L(U, V)) = - \gamma (U) V = 0$. 
From the last relation it follows that $\gamma \big|_{\mathbb{R}^{m}} = 0$. 
Thus, $\gamma = 0$. 
We conclude that  $L(\mathscr{R} (\mathfrak{g})) \subset \mathfrak{so} (1, n + 1)$, and $\mathfrak{g}$ is not a Berger algebra.

\textbf{Case b.1.} 
Rewrite $\theta : \mathbb{R} \oplus \mathfrak{h} \rightarrow \mathbb{R}$ in the form $\theta = \theta_1 \oplus \theta_2$, where $\theta_1 = \theta \big|_\mathbb{R}$ and $\theta_2 = \theta \big|_\mathfrak{h}$. 
Then, 
$$\mathfrak{g} = \big( \{ (\theta_1 (a), a, 0, 0) \mid a \in \mathbb{R} \} \oplus \{ (\theta_2 (A), 0, A, 0 ) \mid A \in \mathfrak{h} \} \big) \ltimes \mathbb{R}^{n}.$$ 
Choose the algebraic tensor $R$ defined by the condition $\lambda = 1$, $\mu = \theta_1 (1)$ and all other elements are zeros. 
Since ${R \in \mathscr{R} (\mathfrak{g})}$, then $\{ (\theta_1 (a), a, 0, 0) \mid a \in \mathbb{R} \} \subset L(\mathscr{R} (\mathfrak{g}) )$.

We will choose $R$ as follows. Take $P \in \mathscr{P} (\mathfrak{h})$ and define $\gamma (U) := \theta_2 (P(U))$. 
Choose $X_0$ such that $(U, X_0) = \gamma (U)$. 
We suppose all other elements from definition of $R$ to be zero. 
According to Theorem~\ref{prsoTh}, $\mathfrak{h}$ is the holonomy algebra of a Riemannian manifold, and, hence, it is a weak Berger algebra~\cite{Leistner07}. 
Therefore, elements $P(U)$ generate $\mathfrak{h}$, at the same time 
$$R(U, q) = (\theta_2 (C(U)), 0, C(U), 0 ).$$
Going through all $P \in \mathscr{P} (\mathfrak{h})$ and $U \in \mathbb{R}^{n}$ we get $\{ (\theta_2 (A), 0, A, 0 ) \mid A \in \mathfrak{h} \}$. 
Hence $\mathfrak{g} \subset L(\mathscr{R} (\mathfrak{g}))$, i.e., $\mathfrak{g} $ is a Berger algebra. In notation of the statement of Theorem~\ref{class3Th} it is the Lie algebra  $\mathfrak{g}^{\alpha, \theta, 1,\mathfrak{h}}$ with $\theta = \theta_2$, $\alpha = \theta_1 (1)$.

\textbf{Case b.2.} 
Assume that $R \in \mathscr{R} (\mathfrak{g})$. One can represent $\theta : \mathfrak{h} \oplus \mathbb{R}^{n_0} \rightarrow \mathbb{R}$ in the form $\theta = \theta_1 \oplus \theta_2$, where $\theta_1 = \theta \big|_\mathfrak{h}$ and $\theta_2 = \theta \big|_{\mathbb{R}^{n_0}}$. 
Then, 
$$\mathfrak{g} = \big( \{ (\theta_1 (A), 0, A, 0) \mid A \in \mathfrak{h} \} \oplus \{ (\theta_2 (X), 0, 0, X ) \mid X \in \mathbb{R}^{n_0} \} \big) \ltimes (\mathbb{R}^{n_0})^{\perp}.$$ 
We will show, that $\mathfrak{g}$ is a Berger algebra only if $\theta_2 = 0$.

Suppose that $\mathfrak{g}$ is a Berger algebra and $\theta_2 \neq 0$. 
Then there exists an element $V \in \mathbb{R}^{n_0}$ such that $\theta_2 (V) \neq 0$. From the equality $\theta_2 (-\mu V) = 0$ it follows that $\mu = 0$. From the expressions for $R(p,q)$, $R(U,V)$ and $R(U,q)$ we obtain the following relations:
\begin{gather} 
\label{b2eq1} \theta_2 (X_0) = 0,\\
\label{b2eq2} \gamma (U) = (U, X_0),\\
\label{b2eq3} \theta_2 (L(U, V)) = 0.
\end{gather}
If $\theta_2 \neq 0 $, then $\dim \ker \theta_2 = n_0 - 1$. 
Choose a basis $e_1, \ldots, e_{n_0}, \ldots , e_n$ such that $\theta_2 (e_1) \neq 0$ and $\theta_2 (e_i) = 0$ for $i \geqslant 2$.	
Due to the formulas~\eqref{b2eq1} and~\eqref{b2eq2}, we have $X_0 \in \Span \{ e_2, \ldots , e_{n_0} \}$ and $\gamma (e_1) = (e_1, X_0) = 0$. 
Substituting to the formula~\eqref{b2eq3} $U = e_1$ and $V = e_i$ for $i \geqslant 2$, we obtain 
$$\theta_2 (L(e_1, e_i)) = \theta_2 (\gamma (e_i) e_1) = \gamma (e_i) \theta_2 (e_1) = 0.$$ 
Since $\theta_2 (e_1) \neq 0$, then $\gamma (e_i) = 0$. So, $\gamma = 0$ and, hence, $\mathfrak{g} $ is not a Berger algebra whenever $\theta_2 \neq 0$. When $\theta_2 = 0$, the algebra $\mathfrak{g} $ is a Berger algebra, since $(\theta (A), 0, A, 0 ) \subset L(\mathscr{R} (\mathfrak{g}) )$.

\textbf{Case b.3.} 
The proof is similar as for the case b.1.

\textbf{Case b.4.} We assume that $\bar{\mathfrak{g}}$ is of type 4.
It is clear that $L(\mathscr{R}(\mathfrak{g}))\subset L(\mathscr{R}(\mathbb{R}\id_{\mathbb{R}^{1,n+1}}\oplus\bar{\mathfrak{g}}))$.
Above in Case a.4 we have seen that $L(\mathscr{R}(\mathbb{R}\id_{\mathbb{R}^{1,n+1}}\oplus\bar{\mathfrak{g}}))\subset\mathfrak{so}(1,n+1)$. This shows that $\mathfrak{g}$ is not a Berger algebra.

Theorem~\ref{class3Th} is proved.\hfill\qedsymbol

\section{Realization of Berger algebras}
\label{RealizationnSec}

In this section we prove Theorem~\ref{realizationTh}. 
To do this we show that for every Berger algebra $\mathfrak{g} \subset \mathfrak{co}(1, n+1)$ obtained above there exists a Weyl connection $\nabla$ 
such that the holonomy algebra of $\nabla$ is isomorphic to $\mathfrak{g}$.

Let $(M,c,\nabla)$ be a Weyl manifold. Fix a metric $g\in c$. There exists a 1-form $\omega$ such that $$\nabla g = 2 \omega \otimes g.$$
 Let $\overline{\nabla}$ be the Levi-Civita connection corresponding to $g$.
 It is well-known that the connection $\nabla$ is uniquely defined by $g$ and $\omega$ through the formulas 
$$\nabla = \overline{\nabla} + K,$$
$$g(K_X(Y),Z)=g(Y,Z)\omega(X)+g(X,Z)\omega(Y)-g(X,Y)\omega(Z).$$
In what follows we will assume that $g$ is a Walker metric on $\mathbb{R}^{n+2}$.
This means that 
we fix coordinates $v,x^1,\dots,x^n,u$ and set
\begin{equation}\label{metric} g=2dvdu+h+2 Adu +H (du)^2,\end{equation}
where $$h=
h_{ij}dx^idx^j,\quad\partial_vh_{ij}=0,$$ is a $u$-family of Riemannian metrics on $\mathbb{R}^n$, $$A=A_idx^i$$ is a 1-form, $\partial_v A_i=0$, and $H$ is an arbitrary function. This metric has the property that the vector field $\partial_v$ generates a $\overline{\nabla}$-parallel distribution of isotropic lines. We would like that this distribution is parallel also with respect to the connection $\nabla$. This is guarantied by the condition
$$\omega=f du$$
for some function $f$.   
Then $K$ is given by \begin{align}
\label{KEq}
\begin{split}
K &= f \cdot (\partial_v \wedge \cdot + g (\partial_v, \cdot) \id_{\mathbb{R}^{n+2}})\quad \text{i.e.}\\
K_{\partial_a} \partial_b &= f \cdot ((\partial_v \wedge \partial_a) \partial_b + g (\partial_v, \partial_a) \partial_b).
\end{split}
\end{align}
Denote by $\Gamma_{a}$ the matrix $(\Gamma^{c}_{ba})_{b,c = v,1,\ldots , n, u}$ of the Christoffel symbols. For the Christoffel symbols of the connection $\nabla$ it holds 
$$\Gamma_a = \overline{\Gamma}_a + K_a, \quad \overline{\Gamma}_a = (\overline{\Gamma}^{c}_{ba})_{b,c = v,1,\ldots , n, u},  \quad K_a = (K^{c}_{ba})_{b,c = v,1,\ldots , n, u},\quad K_{\partial_a} \partial_b = K^{c}_{ba} \partial_c,$$
where $\overline{\Gamma}_a$ are the Christoffel symbols for the Levi-Civita connection $\overline{\nabla}$ of the metric~$g$ and $K_a$ is given by the formula~\eqref{KEq}.

{\bf Realization of algebras from Theorem \ref{class2Th}.}
The first algebra represents the most general case of the conformal products in the sense of \cite{Belgun11} and   can be realized in the same way as in~\cite{Grabbe14}. 
Now, suppose that $\mathfrak{g}$ is the second or the third algebra from Theorem~\ref{class2Th}.

First of all consider  the metric
\begin{equation}
\label{metricEq2}
g = 2 d v d u +h+H(du)^2,\quad h=\sum_{i,j = 1}^{n} h_{ij} d x^i d x^j,
\end{equation}
$$h_{ij} = h_{ij}(x^1,\ldots , x^n, u),\quad H=H(v).$$
Let $\overline{\nabla}$ be the Levi-Civita connection of $g$, and let $$\omega=fdu$$ for some function $f$. As it is explained above, $g$ and $\omega$ define a Weyl connection $\nabla$.
The Christoffel symbols for the connection $\nabla$ are as follows
\begin{align*}
\label{GammaEq2}
\Gamma_v &= \begin{pmatrix} 
0 & 0 & \frac{1}{2} \partial_v H \\
0 & 0 & 0 \\
0 & 0 & 0 
\end{pmatrix}, \\
\Gamma_i &= 
\begin{pmatrix} 
0 & \left( - \frac{1}{2} \partial_u h_{i1},\ldots,- \frac{1}{2} \partial_u h_{in} \right) & 0 \\
0 & \widetilde{\Gamma}_i & \left( \left( \frac{1}{2} h^{km} \partial_u h_{im} \right)^{n}_{k = 1} \right)^{t} \\
0 & 0 & 0 
\end{pmatrix} + f 
\begin{pmatrix} 
0 & (-h_{i1},\ldots,-h_{in}) & 0 \\
0 & 0 & e_i \\
0 & 0 & 0 
\end{pmatrix}, \\
\Gamma_u &= 
\begin{pmatrix} 
\frac{1}{2}\partial_v H & 0 & \frac{1}{2} H \partial_v H \\
0 & \left(  \frac{1}{2} h^{km} \partial_u h_{lm} \right)_{k, l = 1,\ldots,n} & 0 \\
0 & 0 & -\frac{1}{2}\partial_v H
\end{pmatrix} + f 
\begin{pmatrix} 
0 & 0 & -H \\
0 & E_n & 0 \\
0 & 0 & 2 
\end{pmatrix},
\end{align*}
where $\widetilde{\Gamma}_i$ are the Christoffel symbols of $h$ and $(h^{ij})_{i, j = 1,\ldots, n}$ is the inverse of the matrix $(h_{ij})_{i, j = 1,\ldots, n}$. 

Let $(N,h)$ be a Riemannian manifold of dimension $k$ with
the holonomy algebra $\mathfrak{k}\subset\mathfrak{so}(k)$.
Consider the manifold
$$M=\mathbb{R}\times N\times \mathbb{R}^{n-k}\times \mathbb{R}.$$
We will consider coordinates $v,x^1,..., x^k,x^{k+1},...,x^n,u$, where 
$x^1,..., x^k$ are local coordinates on $N$. We choose the function $f$  as
$$f=x^{k+1}.$$
Let $$h=h^1+h^2,$$
where $$h^2=\sum_{i,j = k+1}^{n} h^2_{ij} d x^i d x^j,\quad h^2_{ij} = e^{-2fu}\delta_{ij}.$$
We obtain the Weyl manifold $(M, [g], \nabla)$.
The above formulas for the Christoffel symbols show that the distribution generated by the vector fields $\partial_{x^{k+1}},...,\partial_{{x^n}}$ is parallel. Similarly, the distribution generated by the vector field 
$\partial_{v}$ is parallel.

Fix a point $x_0\in N$. Consider the point $a\in M$ defined by $x_0$ and $0\in  
\mathbb{R}\times \mathbb{R}^{n-k}\times \mathbb{R}$.
We identify the tangent space $T_aM$ with the Minkowski space $\mathbb{R}^{1,n+1}$ and introduce a Witt basis $p,e_1,...,e_n,q$ in such a way that $e_1,...,e_k$ is an orthonormal basis of $T_{x_0}N$, $p=(\partial_v)_a$, 
$e_{k+1}=(\partial_{x^{k+1}})_a$,...,$e_{n}=(\partial_{x^{n}})_a$, $q=(\partial_u)_a$. We see that the holonomy algebra $\mathfrak{hol}_a(\nabla)$
preserves the isotropic line $\mathbb{R}p$ and the vector subspace $\mathbb{R}^{n-k}\subset\mathbb{R}^{1,n+1}$.  
The direct simple computations show that 
$$R (\partial_{k+1}, \partial_u)_{a} = 
\begin{pmatrix} 
0 & 0 & 0\\
0 & E_n & 0\\
0 & 0 & 2
\end{pmatrix}, 
\hphantom{R (\partial_{k+1}, \partial_u)_{0}}.$$ This implies that 
the Weyl structure is non-closed. From the classification Theorem~\ref{class2Th}  it follows that the  holonomy algebra $\mathfrak{hol}_a(\nabla)$ is one of the following algebras listed in the theorem: the first algebra with $k=0$, the second algebra, or the third algebra.

Assume that $H=0$. We see that the vector field $\partial_v$ is parallel, and the holonomy algebra annihilates the vector $p$. We conclude that $\mathfrak{hol}_a(\nabla)$ is the second algebra from Theorem~\ref{class2Th}.

Finally let $H=v^2$. In this case it holds 
$$R (\partial_{v}, \partial_u)_{a} = 
\begin{pmatrix} 
1 & 0 & 0\\
0 & 0 & 0\\
0 & 0 & -1
\end{pmatrix}.$$
This shows that  $\mathfrak{hol}_a(\nabla)$ acts non-trivially on the line $\mathbb{R}p$. If $k=0$, then $\mathfrak{hol}_a(\nabla)$ is the first algebra from Theorem \ref{class2Th}. If $k\geq 1$, then $\mathfrak{hol}_a(\nabla)$ is the last algebra Theorem \ref{class2Th}.

{\bf Realization of algebras from Theorem \ref{class3Th}.}
Let $\mathfrak{g}\subset\mathfrak{co}(1,n+1)_{\mathbb{R}p}$ be one of the algebras listed  in  Theorem \ref{class3Th}.
Let $\mathfrak{h}\subset\mathfrak{so}(n)$ be the corresponding subalgebra, which is the holonomy algebra of a Riemannian manifold. First we fix our attention on some properties of $\mathfrak{h}$. 

According to the de~Rham Theorem, there exists an orthogonal decomposition 
$$\mathbb{R}^{n} = \mathbb{R}^{n_0} \oplus \mathbb{R}^{n_1} \oplus \ldots \oplus \mathbb{R}^{n_r}$$ 
and the corresponding decomposition
$$\mathfrak{h} = \{0\} \oplus \mathfrak{h}_1 \oplus \ldots \oplus \mathfrak{h}_r$$ 
into the direct sum of ideals such that $\mathfrak{h}$ acts trivially on $\mathbb{R}^{n_0}$, $\mathfrak{h}_i (\mathbb{R}^{n_j}) = 0$ for $i \neq j$, and $\mathfrak{h}_i \subset \mathfrak{so}(n_i)$ is an irreducible subalgebra for $1 \leqslant i \leqslant r$. 
Moreover, each Lie algebra $\mathfrak{h}_i$ is the holonomy algebra of a Riemannian manifold. It is known~\cite{Gal05} that it holds
$$\mathscr{P}(\mathfrak{h}) = \mathscr{P}(\mathfrak{h}_1) \oplus \ldots \oplus \mathscr{P}(\mathfrak{h}_r).$$ 
We will assume that the basis 
of $\mathbb{R}^{n}$ is compatible with this decomposition of $\mathbb{R}^{n}$.
We will need the following lemma from~\cite{Gal15}.

\begin{lemma}
	\label{PRnLemma} 
	For an arbitrary holonomy algebra $\mathfrak{h}\subset\mathfrak{so}(n)$ of a Riemannian manifold there exists a $P\in\mathscr{P}(\mathfrak{h})$ such that the vector space ${P(\mathbb{R}^n)\subset\mathfrak{h}}$ generates $\mathfrak{h}$ as  
	Lie algebra.
\end{lemma}

Fix a $P\in\mathscr{P}(\mathfrak{h})$ as in the lemma and define the matrices $P_{i} = (P_{ji}^{k})_{j,k = 1,\ldots , n}$ such that $P(e_i) e_j = P_{ji}^{k} e_k$.

Let $v, x^1, \ldots, x^n, u$ be the standard coordinates on $\mathbb{R}^{n + 2}$.
Consider the metric $g$ given by the formula
\begin{equation}
\label{metricEq}
g = 2 dv du + \sum_{i = 1}^{n} (dx^i)^2 + 2 \sum_{i = 1}^{n} A_i dx^i du + H \cdot (du)^2,
\end{equation}
where
$$
A_i = \frac{1}{3} (P_{jk}^{i} + P_{kj}^{i}) x^j x^k
$$
and $H$ is a function that will depend on the type of the holonomy algebra that we wish to obtain. Let $\overline{\nabla}$ be the Levi-Civita connection corresponding to $g$.
Let  $p, e_1, \ldots , e_n, q$ be the field of frames 
$$p = \partial_v,\quad e_i = \partial_i - A_i \partial_v,\quad q = \partial_u - \frac{1}{2} H \partial_v,$$
it defines a Witt basis of the tangent space at each point of $\mathbb{R}^{n + 2}$. Consider the 1-form
$$\omega=fdu,$$ where $f$ is a function that will be specified just below.
As it is explained above, the metric $g$ and the 1-form $\omega$ define the Weyl manifold $(\mathbb{R}^{n+2},[g],\nabla)$, $\nabla=\overline{\nabla}+K$.

For the Lie algebras $\mathbb{R} \id_{\mathbb{R}^{1,n+1}} \oplus \mathfrak{g}^{3,\mathfrak{h},\varphi}$ and $\mathfrak{g}^{\theta,3,\mathfrak{h},\varphi}$ define the numbers $\varphi_i = \varphi (P(e_i))$. 
For the Lie algebras $\mathfrak{g}^{\alpha, \theta, 1,\mathfrak{h}}$, $\mathfrak{g}^{\theta,2,\mathfrak{h}}$ and $\mathfrak{g}^{\theta,3,\mathfrak{h},\varphi}$ define the numbers $\theta_i = \theta (P(e_i))$. 

The following theorem shows that each algebra from Theorem~\ref{class3Th} is the holonomy algebra.

\begin{theorem}
	\label{realizTabTh}
	The holonomy algebra $\mathfrak{hol}_{0}(\nabla)$ of the Weyl connection $\nabla$ depends on the functions $H$ and $f$ in the following way:
\end{theorem}
\begin{table}[H]
	\centering
	\label{tab1}
	\begin{tabular}{|c|c|c|}
	\hline $H$ & $f$ & $\mathfrak{hol}_{0}(\nabla)$ \\
	\hline $\frac{1}{3} v^3 + \sum_{i = 1}^{n_0} (x^i)^2$ & $v$ & $\mathbb{R} \id_{\mathbb{R}^{1,n+1}} \oplus \mathfrak{g}^{1,\mathfrak{h}}$ \\
	\hline $v^2 + \sum_{i = 1}^{n_0} (x^i)^2$ & $v$ & $\mathbb{R} \id_{\mathbb{R}^{1,n+1}} \oplus \mathfrak{g}^{2,\mathfrak{h}}$ \\
	\hline $v^2 + 2 \sum_{i = n_0 + 1}^{n} \varphi_i x^i v + \sum_{i = 1}^{n_0} (x^i)^2$ & $v$ & $\mathbb{R} \id_{\mathbb{R}^{1,n+1}} \oplus \mathfrak{g}^{3,\mathfrak{h},\varphi}$ \\
	\hline $(1 + \alpha ) v^2 + 2 \sum_{i = n_0 + 1}^{n} \theta_i x^i v + \sum_{i = 1}^{n_0} (x^i)^2$ & $\alpha v + \sum_{i = n_0 + 1}^{n} \theta_i x^i$ & $\mathfrak{g}^{\alpha,\theta,1,\mathfrak{h}}$ \\
	\hline $2 \sum_{i = n_0 + 1}^{n} \theta_i x^i v + \sum_{i = 1}^{n_0} (x^i)^2$ & $\sum_{i = n_0 + 1}^{n} \theta_i x^i$ & $\mathfrak{g}^{\theta,2,\mathfrak{h}}$ \\
	\hline $2 \sum_{i = n_0 + 1}^{n} (\theta_i + \varphi_i) x^i v + \sum_{i = 1}^{n_0} (x^i)^2$ & $\sum_{i = n_0 + 1}^{n} \theta_i x^i$ & $\mathfrak{g}^{\theta,3,\mathfrak{h},\varphi}$ \\
	\hline
	\end{tabular}
\end{table}

\begin{remark*}
	Note that the projection of $\mathfrak{hol}_{0}(\nabla)$ to $\mathfrak{so} (1, n + 1)$ may not coincide with $\mathfrak{hol}_{0}(\overline{\nabla})$.
\end{remark*}

\begin{proof}[Proof of Theorem~\ref{realizTabTh}]
Above we fixed an algebra $\mathfrak{g}$ from Theorem~\ref{class3Th} and constructed using it the Weyl connection $\nabla$. Now we are going to prove that  the holonomy algebra of that connection coincides with $\mathfrak{g}$.
Since the metric $g$ and the function $f$ are polynomial, the connection $\nabla$ is analytic and from the proof of Theorem~3.9.2 in~\cite{Kobayashi} it follows that the holonomy algebra $\mathfrak{hol}_{0} (\nabla)$ of the connection $\nabla$ at the point $0$ is generated by the elements of the form
$$ \nabla_{Z_\alpha} \cdots \nabla_{Z_1}(R(X, Y))_{0} \in \mathfrak{co} (1, n + 1),
\quad X, Y, Z_1,\ldots , Z_\alpha \in T_0 M,
\quad \alpha=0,1,2,\ldots .$$
We have the following recursion formula
\begin{equation}
\label{RecurrentEq}
\nabla_{a_\alpha}\cdots\nabla_{a_1} (R (\partial_a, \partial_b)) =
\partial_{a_\alpha}\nabla_{a_{\alpha-1}}\cdots\nabla_{a_1} (R (\partial_a, \partial_b))
+ [\Gamma_{a_\alpha}, \nabla_{a_{\alpha - 1}}\cdots\nabla_{a_1} (R (\partial_a, \partial_b))].
\end{equation}

It is easy to check directly that the Christoffel symbols of the connection $\nabla$ are the following:
\begin{align*}
\label{GammaEq}
\Gamma_v &= \begin{pmatrix} 
0 & 0 & \frac{1}{2} \partial_v H \\
0 & 0 & 0 \\
0 & 0 & 0 \end{pmatrix}, \quad
\Gamma_i = \begin{pmatrix} 
0 & \ast & \ast \\
0 & 0 & \frac{1}{2} F_{i} + ((\delta_{ik} f)^{n}_{k = 1})^{t} \\
0 & 0 & 0 \end{pmatrix}, \\
\Gamma_u &= \begin{pmatrix}
\frac{1}{2} \partial_v H & \ast & \ast \\
0 & \frac{1}{2} F + f E_n & - \frac{1}{2} \grad_n H + \frac{1}{2} (\partial_v H) A \\
0 & 0 & - \frac{1}{2} \partial_v H + 2f \end{pmatrix},
\end{align*}
where 
\begin{gather*}
F = (F_{ij})_{i, j = 1,\ldots , n},\quad F_{ij} = \partial_j A_i - \partial_i A_j = 2 P^{i}_{jk} x^k,\quad F_{i} = (F_{1i}, \ldots, F_{ni})^{t}, \\ 
\grad_n H = (\partial_1 H, \ldots, \partial_n H)^{t},
\quad A = (A_{1},\ldots, A_{n})^{t}.
\end{gather*}
We denote by $\ast$ the elements that does not play a role.
Now, using the formula
\begin{gather*}
R(\partial_a, \partial_b) = \partial_a \Gamma_b - \partial_b \Gamma_a + [\Gamma_a, \Gamma_b],
\end{gather*}
one can compute the components of the curvature tensor
\begin{align}
R(\partial_v, \partial_i) &= \begin{pmatrix} 
0 & \ast & \ast \\
0 & 0 &  ((\delta_{ik} \partial_v f)^{n}_{k = 1})^{t} \\
0 & 0 & 0 \end{pmatrix}, 
\\
R(\partial_i, \partial_j) &= \begin{pmatrix} 
0 & \ast & \ast \\
0 & 0 &  (( - P^{j}_{ik} + \delta_{jk} \partial_i f - \delta_{ik} \partial_j f)^{n}_{k = 1})^{t} \\
0 & 0 & 0 \end{pmatrix},
\\
R(\partial_v, \partial_u) &= \begin{pmatrix} 
\frac{1}{2} \partial_v^2 H & \ast & \ast \\
0 & (\partial_v f) E_n & - \frac{1}{2} \partial_v \grad_n H + \frac{1}{2} (\partial_v^2 H) A\\
0 & 0 & - \frac{1}{2} \partial_v^2 H + 2 \partial_v f \end{pmatrix},
\\
\label{RiuEq} R(\partial_i, \partial_u) &= \begin{pmatrix} 
\frac{1}{2} \partial_i \partial_v H & \ast & \ast \\
0 & P_i + (\partial_i f) E_n &  Z_i \\
0 & 0 & - \frac{1}{2} \partial_i \partial_v H + 2 \partial_i f \end{pmatrix},
\end{align}
where $Z_i = ((Z_{ik})_{k=1}^{n})^{t}$ is a vector with the coordinates
\begin{equation*}
Z_{ik} = - \frac{1}{2} \partial_i \partial_k H - \delta_{ik} \partial_u f + (\frac{1}{2} \partial_i \partial_v H) A_k + \frac{1}{2} \partial_v H (\partial_i A_k - \frac{1}{2} F_{ik} - \delta_{ik} f) - \frac{1}{4} \sum_{m = 1}^{n} F_{km} F_{im} + \delta_{ik} f^2.
\end{equation*} 
Also we need the following covariant derivative
\begin{gather}
\nabla_v (R(\partial_v, \partial_u)) = \begin{pmatrix} 
\frac{1}{2} \partial_v^3 H & \ast & \ast \\
0 & (\partial_v^2 f) E_n & - \frac{1}{2} \partial_v^2 \grad_n H + \frac{1}{2} (\partial_v^3 H) A\\
0 & 0 & - \frac{1}{2} \partial_v^3 H + 2 \partial_v^2 f \end{pmatrix}.
\end{gather}~\medskip

At the point $0$, $A_i$ and $H$ vanish, consequently, $\nabla_{Z_\alpha} \cdots \nabla_{Z_1}{(R(Y, Z))}$, $\, Y, Z, Z_1,\ldots , Z_\alpha \in T_0 M$ are given by the matrices of the form
$$\begin{pmatrix}     
a + b & X^t & 0 \\
0 & B + b E_n & -X \\
0 & 0 & -a + b
\end{pmatrix},
\quad a, b \in \mathbb{R},
\quad B \in \mathfrak{so} (n),
\quad X \in \mathbb{R}^n,$$
and as above may be written as $(b, a, B, X)$.

\noindent\textbf{Proof of the inclusion $\mathfrak{g} \subset \mathfrak{hol}_{0}$.}
\begin{lemma}
	\label{RnHolLemma}
	It holds $\mathbb{R}^{n} \subset \mathfrak{hol}_{0}(\nabla)$.
\end{lemma}

\begin{proof}
	If $1 \leqslant i \leqslant n_0$, then $R(\partial_i, \partial_u)_{0} = \left( 0, 0, 0, \frac{1}{2} e_i \right)$, hence, $\mathbb{R}^{n_0} \subset \mathfrak{hol}_{0}(\nabla)$. 	
	We will show that $\mathbb{R}^{n_\alpha} \subset \mathfrak{hol}_{0}(\nabla)$ for $1 \leqslant \alpha \leqslant r$. 
	Assume that $e_i, e_j, e_k \in \mathbb{R}^{n_\alpha}$. Then
	$$R(\partial_i, \partial_j)_{0} = \left( 0, 0, 0, \sum_{e_k \in \mathbb{R}^{n_\alpha}} P^{j}_{ik} e_k -  s_i e_j + s_j e_i \right),$$
	where $s_i = (\partial_i f)(0)$.
	We claim that there exist $i, j$ such that
	$$ \sum_{e_k \in \mathbb{R}^{n_\alpha}} P^{j}_{ik} e_k - s_i e_j + s_j e_i \neq 0.$$
	For the first three connections from the statement of the theorem it holds $s_i = 0$ and there is nothing to prove. 
	In the rest three cases it holds $s_i = \theta_i$. 
	Suppose that $P^{j}_{ik} = s_i \delta_{jk} - s_j \delta_{ik}$ and, hence, 
	$$P_{k} = \sum_{e_i \in \mathbb{R}^{n_\alpha}} s_i e_i \wedge e_k.$$	
	Since $P \big|_{\mathbb{R}^{n_\alpha}} \neq 0$, then the only possible option is that $\mathfrak{h}_\alpha = \mathfrak{so} (n_\alpha)$.
	For $n_\alpha \geqslant 3$ it holds $\theta \big|_{\mathfrak{h}_\alpha} = 0$ and we have a contradiction. For $n_\alpha = 2$ we will assume for convenience that $e_1, e_2$ is a basis of $\mathbb{R}^{n_\alpha}$.
	In this case
	$$s_1 = \theta (P_1) = - s_2 \xi, \quad s_2 = \theta (P_2) = s_1 \xi, \quad \text{where} \quad \xi = \theta \left( \begin{pmatrix} 
	0 & - 1 \\
	1 & 0  \end{pmatrix} \right).$$
	Therefore $s_1 = - s_1 \xi^2$ and hence $s_1 = s_2 = 0$.
	
	Thus there exists non-zero $Y \in \mathbb{R}^{n_\alpha}$, such that $\big( 0, 0, 0, Y \big) \in \mathfrak{hol}_{0}(\nabla)$. Since the image of $P \big|_{\mathbb{R}^{n_\alpha}}$ generates the Lie algebra $\mathfrak{h}_\alpha$ and $\mathfrak{hol}_{0}(\nabla)$ is a Lie algebra, then from~\eqref{RiuEq} it follows that for every $A \in \mathfrak{h}_\alpha$ there exist $\gamma, \beta$ such that  
	\begin{gather*}
	\zeta := \begin{pmatrix} 
	\gamma + \beta & \ast & \ast \\
	0 & A + \beta E_n & \ast \\
	0 & 0 & - \gamma + \beta \end{pmatrix} \in \mathfrak{hol}_{0}(\nabla).
	\end{gather*}
	Then $[\zeta, (0, 0, 0, Y)] = (0, 0, 0, \gamma Y + A Y) \in \mathfrak{hol}_{0}(\nabla)$ and hence $(0, 0, 0, A Y) \in \mathfrak{hol}_{0}(\nabla)$. Similarly, for every $A_1,\ldots , A_s \in \mathfrak{h}_\alpha$ it holds $(0, 0, 0, A_s \cdots A_1 Y) \in \mathfrak{hol}_{0}(\nabla)$. Since $\mathfrak{h}_\alpha \subset \mathfrak{so} (n_\alpha)$ is irreducible, then $\mathbb{R}^{n_\alpha} \subset \mathfrak{hol}_{0}(\nabla)$. The lemma is proved.
\end{proof}
To complete the proof of the inclusion $\mathfrak{g} \subset \mathfrak{hol}_{0}(\nabla)$ it is necessary to consider each algebra $\mathfrak{g}$ separately. Take, for instance, $\mathfrak{g} = \mathbb{R} \id_{\mathbb{R}^{1,n+1}} \oplus \mathfrak{g}^{1,\mathfrak{h}}$ and let $\nabla$ be the corresponding connection. Then
\begin{align*}
R (\partial_v, \partial_u)_0 & = \big( 1, -1, 0, \ast \big),\\
\nabla_v R (\partial_v, \partial_u)_0 &= \big( 0, \hphantom{-}1, 0, \ast \big).
\end{align*}
From this and Lemma~\ref{RnHolLemma} we obtain $\mathbb{R} \id_{\mathbb{R}^{1,n+1}} \oplus \mathbb{R} (0, 1, 0, 0) \subset \mathfrak{hol}_{0}(\nabla)$. Next, from~\eqref{RiuEq} it holds $(0, 0, P_i, 0) \in \mathfrak{hol}_{0}(\nabla)$. This and Lemma \ref{PRnLemma} imply  $\mathfrak{h} \subset \mathfrak{hol}_{0}(\nabla)$.

For the remaining algebras one can show in the same way that $R (\partial_a, \partial_b)_0$ generate $\mathfrak{g}$.~\medskip

\noindent\textbf{Proof of the inclusion $\mathfrak{hol}_{0} \subset \mathfrak{g}$.}~\medskip

Since $\nabla$ preserves $\langle \partial_v \rangle$, then
$\mathfrak{hol}_{0}(\nabla) \subset \mathfrak{co}(1, n + 1)_{\mathbb{R} p}$. Next, the elements $\Gamma_a$ and $R(\partial_a, \partial_b)$ have the form
\begin{equation}
\label{RFormEq}
\begin{pmatrix} 
\alpha + \beta & \ast & \ast \\
0 & \alpha E_n + B & \ast \\
0 & 0 & \alpha - \beta \end{pmatrix},
\end{equation}
where $\alpha, \beta$ are functions and $B$ is a function with values in $\mathfrak{h}$. Moreover, from~\eqref{RecurrentEq} it follows that every $\nabla_{a_\alpha}\cdots\nabla_{a_1} (R (\partial_a, \partial_b))$ also has the form~\eqref{RFormEq}. Then for each algebra $\mathfrak{g}$ one can check that $R (\partial_a, \partial_b)_0 \in \mathfrak{g}$ and $\nabla_{a_\alpha}\cdots\nabla_{a_1} (R (\partial_a, \partial_b))_0 \in \mathfrak{g}$.

The theorem is proved.
\end{proof}



\end{document}